\documentclass[a4paper,reqno]{amsart}
\usepackage{amssymb,amscd,verbatim}
\usepackage{rotating,lscape,enumerate}
\usepackage[mathscr]{euscript}
\usepackage[french,english]{babel}
\usepackage[T1]{fontenc}
\usepackage{lmodern}
\usepackage{xspace}
\usepackage{xcolor}
\usepackage[colorlinks=true,linkcolor=blue,urlcolor=violet,citecolor=red]{hyperref}

\newtheorem{thm}{Theorem}[section]
\newtheorem{lem}[thm]{Lemma}
\newtheorem{prop}[thm]{Proposition}
\newtheorem{cor}[thm]{Corollary}

\newtheorem*{theo}{Theorem}

\theoremstyle{definition}
\newtheorem{defn}[thm]{Definition}

\newtheorem*{notation}{Notation}

\newtheorem{rem}[thm]{Remark}
\newtheorem{rems}[thm]{Remarks}
\newtheorem{conj}[thm]{Conjecture}

\theoremstyle{remark}
\newtheorem*{remark}{Remark}
\newtheorem*{remarks}{Remarks}

\newtheorem*{examples}{Examples}

\theoremstyle{plain}
\newtheorem{mainthm}{Theorem}

\newtheorem{mainhyp}[mainthm]{Hypothesis}

\numberwithin{equation}{section}

\newenvironment{french}{\begin{otherlanguage}{french}}{\end{otherlanguage}}

                       {\end{itemize}}

                       {\end{enumerate}}

\newcommand{\Q}{{\mathbb{Q}}}
\newcommand{\N}{{\mathbb{N}}}
\newcommand{\C}{{\mathbb{C}}}
\newcommand{\R}{{\mathbb{R}}}
\newcommand{\D}{{\mathbb{D}}}
\newcommand{\tD}{{\widetilde{\mathbb{D}}}}
\newcommand{\Z}{{\mathbb{Z}}}

\newcommand{\bbS}{{\mathbb{S}}}
\newcommand{\bbtS}{{\widetilde{\mathbb{S}}}}

\newcommand{\calC}{{\mathcal{C}}}
  \newcommand{\cO}{{\mathcal{O}}}

\newcommand{\calA}{{\mathcal{A}}}
\newcommand{\calB}{{\mathcal{B}}}
\newcommand{\calD}{{\mathcal{D}}}  \newcommand{\cD}{{\mathcal{D}}}

\newcommand{\calF}{{\mathcal{F}}}  
  
  \newcommand{\cH}{{\mathcal{H}}}

\newcommand{\calL}{{\mathcal{L}}}  
\newcommand{\calM}{{\mathcal{M}}}  \newcommand{\cM}{{\mathcal{M}}}
\newcommand{\calN}{{\mathcal{N}}}  \newcommand{\cN}{{\mathcal{N}}}

\newcommand{\calS}{{\mathcal{S}}}

\DeclareMathAlphabet{\euls}{U}{eus}{m}{n}

\newcommand{\fa}{{\mathfrak{a}}}
\newcommand{\fc}{{\mathfrak{c}}}
\newcommand{\g}{{\mathfrak{g}}}  \renewcommand{\fg}{{\mathfrak{g}}}
\newcommand{\fh}{{\mathfrak{h}}} \newcommand{\h}{{\mathfrak{h}}}

\newcommand{\fp}{{\mathfrak{p}}} 

\newcommand{\fk}{{\mathfrak{k}}}  \renewcommand{\k}{{\mathfrak{k}}}

\newcommand{\fsl}{{\mathfrak{sl}}}

\newcommand{\fs}{{\mathfrak{s}}}
\newcommand{\ft}{{\mathfrak{t}}}

\newcommand{\Xt}{{\mathtt{X}}}
\newcommand{\Bt}{{\mathtt{B}}}
\newcommand{\Rt}{{\mathtt{R}}}

\newcommand{\tti}{{\mathtt{i}}}
\newcommand{\ttj}{{\mathtt{j}}}

\newcommand{\esf}{{\mathsf{e}}}
\newcommand{\msf}{{\mathsf{m}}}
\newcommand{\asf}{{\mathsf{a}}}

\DeclareMathSymbol{\Ima}{\mathord}{symbols}{"3D}

\def\preisomto{\vbox{\hbox to
               14pt{\hfill$\sim$\hfill}\nointerlineskip\vskip -0.2pt
               \hbox to 14pt{\rightarrowfill}}}
\def\isomto{\mathop{\preisomto}}

\def\prelongisomto{\vbox{\hbox to
                17pt{\hfill$\sim$\hfill}\nointerlineskip\vskip -0.2pt
                \hbox to 17pt{\rightarrowfill}}}  
\def\longisomto{\mathop{\prelongisomto}}

\def\preWedge{\vbox{\hbox
               {\hfill$\bigwedge$\hfill}\nointerlineskip\vskip -3.2pt
               \hbox {\phantom{.}}}}
\def\Wedge{\mathop{\preWedge}}

\newcommand{\stoo}{\longrightarrow \kern-15pt
\longrightarrow}
\newcommand{\lto}{\longrightarrow}

\newcommand{\sto}{\twoheadrightarrow}

\def\trait{\hbox to 12.6cm{\hrulefill}}

\newcommand{\imply}{\Rightarrow}
\newcommand{\limply}{\Longrightarrow}

\def\ad{\operatorname {ad}}

\def\ann{\operatorname {ann}}
\def\Ker{\operatorname {Ker}}
\def\Im{\operatorname {Im}}
\def\Hom{\operatorname {Hom}}

\def\mod{\operatorname {\mathsf{mod}}}
\def\End{\operatorname {End}}

\def\GL{\operatorname {GL}}
\def\SL{\operatorname {SL}}
\def\Sp{\operatorname {Sp}}
\def\SO{\operatorname {SO}}

\def\gr{\operatorname {gr}}

\def\ch{\operatorname {Ch}}

\def\gk{\operatorname {GKdim}}

\def\Lie{\operatorname {Lie}}

\def\Sym{\operatorname {Sym}}
\def\Sol{\operatorname {Sol}}

\def\rad{\operatorname {rad}}

\def\Mat{\operatorname {M}}

\def\SO{\operatorname {SO}}

\def\Sp{\operatorname {Sp}}

\def\res{\operatorname {res}}
\def\Frac{\operatorname {Frac}}
\def\ord{\operatorname {ord}}

\def\boplus{\mathbin{\boldsymbol{\oplus}}}

\newcommand{\Exp}{\mathrm{e}}

\newcommand{\id}{{\mathrm{id}}}

\newcommand{\dual}[2]{{\langle {#1} ,  {#2} \rangle}}

\newcommand{\ascal}[2]{{\langle {#1}  \mid {#2} \rangle}}

\newcommand{\vpi}{\varpi}

\newcommand{\vphi}{\varphi}

\newcommand{\tG}{{\tilde{G}}}

\newcommand{\tT}{{\tilde{T}}}
\newcommand{\tLambda}{{\tilde{\Lambda}}}
\newcommand{\tGamma}{{\tilde{\Gamma}}}
\newcommand{\tgamma}{{\tilde{\gamma}}}
\newcommand{\tlambda}{{\tilde{\lambda}}}
\newcommand{\tfg}{{\tilde{\mathfrak{g}}}}
\newcommand{\ftg}{{\tilde{\mathfrak{g}}}}

\newcommand{\tK}{\tilde{K}}
\newcommand{\Vtlambda}{{V(\tilde{\lambda})}}
\newcommand{\Vtgamma}{{V(\tilde{\gamma})}}
\newcommand{\Etlambda}{{E(\tilde{\lambda})}}

\newcommand{\Elambda}{{E(\lambda})}

\newcommand{\CV}{{\C[V]}}
\newcommand{\CVG}{{\C[V]^G}}
\newcommand{\SVG}{{S(V)^G}}
\newcommand{\DV}{{\mathcal{D}(V)}}
\newcommand{\Dv}{{\mathcal{D}_V}}
\newcommand{\DVG}{{\mathcal{D}(V)^G}}
\newcommand{\DVmodG}{{\mathcal{D}(V\qmod G)}}
\newcommand{\DVtG}{{\mathcal{D}(V)^{\tilde{G}}}}
\newcommand{\Dh}{{\mathcal{D}(\h)}}

\newcommand{\DhmodW}{{\mathcal{D}(\h/W})}
\newcommand{\Ch}{{\C[\h]}}
\newcommand{\ChW}{{\C[\h]^W}}
\newcommand{\Sh}{{S(\h)}}

\newcommand{\Shs}{{S(\h^*)}}

\newcommand{\SV}{{S(V)}}

\newcommand{\hreg}{{\mathfrak{h}^{\mathrm{reg}}}}
\newcommand{\half}{{\frac{1}{2}}}

\newcommand{\qmod}{{/\!\!/}}

\newcommand{\GV}{{(G : V)}}
\newcommand{\tGV}{{(\tilde{G} : V)}}
\newcommand{\tGVd}{{(\tilde{G} : V^*)}}
\newcommand{\cV}{{\calC(V)}}
\newcommand{\tcV}{{\tilde{\calC}(V)}}
\newcommand{\DVtauG}{{[\DV \tau(\g)]^G}}

\newcommand{\CW}{{\mathbb{C}W}}

\newcommand{\Sph}{{\mathsf{e}\cH\mathsf{e}}}
\newcommand{\sph}{{\mathsf{e}\cH^W}}

\newcommand{\dz}{\partial_z}
\newcommand{\dx}{\partial_x}
\newcommand{\bTheta}{{\bar{\Theta}}}

\newcommand{\Ej}[1]{{E_{#1}(x,\partial_x)}}

\newcommand{\cMP}{{\mathcal{M}_P}}

\begin{document}
% \vspace*{-0.5cm}
%%%%%%%%%%%%%%%%%%%%%%%%%%%%%%%%%%%%%%%%%%%%%%%%%%%%%%%
{\selectlanguage{english}
%%%%%%%%%%%%%%%%%%%%%%%%%%%%%%%%%%%%%%%%%%%%%%%%%%%%%

  \title[Radial components]{Radial Components,
    Prehomogeneous Vector spaces, and Rational Cherednik
    Algebras} \author[T.~Levasseur]{Thierry~Levasseur}
  \address{Laboratoire de Math\'ematiques, UMR6205,
    Universit\'e de Brest, 29238 Brest cedex~3, France}
  \email{Thierry.Levasseur@univ-brest.fr}
  % \thanks{}

\begin{abstract}
  Let $\tGV$ be a finite dimensional representation of a
  connected reductive complex Lie group $\tGV$.  Denote by
  $G$ the derived subgroup of $\tG$ and assume that the
  categorical quotient $V \qmod G$ is one dimensional,
  i.e.~$\CVG = \C[f]$ for a non constant polynomial $f$. In
  this situation there exists a homomorphism $\rad : \DVG
  \to A_1(\C)$, the radial component map, where $A_1(\C)$ is
  the first Weyl algebra. We show that the image of $\rad$
  is isomorphic to the spherical subalgebra of a rational
  Cherednik algebra whose multiplicity function is defined
  by the roots of the Bernstein-Sato polynomial of $f$. In
  the case where $\tGV$ is also multiplicity free we
  describe the kernel of $\rad$ and prove a Howe duality
  result between representations of $G$ occuring in $\CV$
  and lowest weight modules over the Lie algebra generated
  by $f$ and the ``dual'' differential operator $\Delta \in
  S(V)$; this extends results of H.~Rubenthaler obtained
  when $\tGV$ is a parabolic prehomogeneous vector space.
  If $\tGV$ satisfies a Capelli type condition, some
  applications are given to holonomic and equivariant
  $D$-modules on $V$. These applications are related to
  results proved by M.~Muro or P.~Nang in special cases of
  the representation $\tGV$.
\end{abstract}
   
\maketitle

\tableofcontents

% \clearpage

%%%%%%%%%%%%%%%%%%%%%%%%%%%%%
\section{Introduction} \label{sec1}
%%%%%%%%%%%%%%%%%%%%%%%%%%

The base field is the field $\C$ of complex numbers.  Let
$\GV$ be a finite dimensional representation of a connected
reductive Lie group $G$. The action of $G$ extends to
various algebras: $\C[V] = S(V^*)$ the polynomial functions
on $V$, $\DV$ the differential operators on $V$ with
coefficients in $\CV$ and $\SV$ identified with differential
operators on $V$ with constant coefficients.  Recall that
$\DV \cong \CV \otimes \SV$ as a $(\CV,G)$-module and that
$g \in G$ acts on $D \in \DV$ by $(g.D)(\vphi) =
g.D(g^{-1}.\vphi)$ for all $\vphi \in \CV$.  We thus obtain
algebras of invariants $\CVG$, $S(V)^G$ and $\DVG$. Then
$\CVG$ is (by definition) the algebra of regular functions
on the categorical quotient $V \qmod G$ and one can define
the algebra $\cD(V\qmod G)$ of differential operators on
this quotient (see~\cite{EGA} or \cite{MR}).

If $D \in \DVG$ and $f \in \CVG$ one obviously has $D(f) \in
\CVG$; this gives an algebra homomorphism:
$$
\DVG \lto \DVmodG, \ \; D \mapsto \{f \mapsto D(f), \ f \in
\CVG\}.
$$
In general $V\qmod G$ is singular and $\DVmodG$ is difficult
to describe. We will be interested here in the case where
$V\qmod G$ is smooth, i.e.~isomorphic to $\C^\ell$ for some
$\ell \in \N$, in which case $\DVmodG$ is isomorphic to the
Weyl algebra $A_\ell(\C)$.  More precisely, we want to work
with polar representations as defined by J.~Dadoc and V.~Kac
in \cite{DK}. In this case there exists a Cartan subspace
$\h \subset V$, a finite subgroup $W \subset \GL(\h)$
generated by complex reflections ($W \simeq
N_G(\h)/Z_G(\h)$), such that the restriction map $\psi :
\CVG \to \ChW$, $\psi(f) = f_{\mid \h}$, is an
isomorphism. Thus $\psi$ yields the isomorphism $V\qmod G
\isomto \h/W \equiv\C^\ell$ and, consequently, an
isomorphism $\DVmodG \isomto \DhmodW \equiv A_\ell(\C)$.
Recall that among the polar representations one finds two
important classes:

-- the representations with a one dimensional quotient,
i.e.~$\dim V\qmod G =1$;

-- the class of ``theta groups''.

In the latter case there exists a semisimple Lie algebra
$\fs$ and a $\Z_m$-grading $\fs = \boplus_{i=0}^{m-1} \fs_i$
such that $\GV$ identifies with the representation of the
adjoint group of $\fs_0$ acting on $\fs_1$. This generalizes
the case of symmetric pairs $\GV = (K : \fp)$ where (with
obvious notation) $\fs= \fk \boplus \fp$ is the
decomposition associated to a complexified Cartan involution
on $\fs$. Here $\fh \subset \fp$ is a usual Cartan subspace
and $W$ is a Weyl group (cf.~\cite{Hel}).

Return to a general polar representation $\GV$. Combining
the morphism $\DVG \to \DVmodG$ with the isomorphism
$\DVmodG \isomto \DhmodW$ we get the {\em radial component
  map}:
$$
\rad : \DVG \lto \DhmodW, \quad \rad(D)(f) =
\psi(D(\psi^{-1}(f))), \ \; f \in \ChW.
$$
The morphism $\rad$ has proved to be useful in the
representation theory of semisimple Lie algebras, or
symmetric pairs $(\fs : \fk)$ as above, see, e.g.~\cite{Hel,
  Va, LS2, LS3}. Two obvious questions arise: describe the
algebra $R = \Im(\rad) \subset \DhmodW$ and the ideal $J=
\Ker(\rad) \subset \DVG$. Some answers have been given in
particular cases, see for example \cite{LS3,Sch2}, and it is
expected that the algebra $\DVG/J$ has a representation
theory similar to that of factors of enveloping algebras of
semisimple Lie algebras (cf.~\cite{Sch3}).

It is known that in the case $\GV = (K : \fp)$ of a
symmetric pair, the subalgebra $\rad\bigl(S(\fp)^K\bigr)$ of
$R$ can be described via the introduction Dunkl operators
\cite{Du, Heck, deJ, To1}. It is therefore natural to use
rational Cherednik algebras \cite{DO, EG, Gor} to describe
$R$. Recall that to each complex reflection group $(W :
\fh)$ is associated an algebra $\cH(k)$ where $k$ is a
``multiplicity function'' on the set of reflecting
hyperplanes in $\h$. Denoting by $\hreg$ the complement of
these hyperplanes, $\cH(k)$ is a subalgebra of the crossed
product $\cD(\hreg) \rtimes \CW$ generated by $\Ch$, $\CW$
and a subalgebra $\C[T_1\dots,T_\ell] \cong S(\h)$ where
each $T_i$ is a (generalized) Dunkl operator, see
\S\ref{ssec20} for details. If $\esf = \frac{1}{|W|} \sum_{w
  \in W} w \in \CW$ is the trivial idempotent, $\esf \cH(k)
\esf$ is called the spherical subalgebra. Then one can show
that there exists an injective homomorphism
$$
\res : \esf \cH(k) \esf \lto \cD(\fh/W)
$$
and we obtain in this way a family $U(k) = \res(\esf \cH(k)
\esf)$ of subalgebras of $\cD(\fh/W)$. One would like to
obtain information on $R$ by answering the following
question:
\begin{center}
  {\sl Does there exist a multiplicity function $k$ such that $R
  = U(k)$?}
\end{center}

For instance, suppose that $\GV= (K : \fp)$ as above.  The
reflecting hyperplanes are then parametrised by elements of
the reduced root system $\Rt$ defined by $(\fs, \h)$ and one
defines a multiplicity function by:
$$
k(\alpha) = \half\bigl(\dim \fs^\alpha + \dim
\fs^{2\alpha}\bigr), \ \; \alpha \in \Rt,
$$
where $\fs^\beta$ is the root space associated to the root
$\beta$. For this choice of $k$ one can prove \cite{LSx}:

\begin{theo} [L--Stafford]
  One has $R = \Im(\rad) = U(k) \cong \esf \cH(k) \esf $.
\end{theo}

Our aim in this work is to analyse a simpler case, $G$
semisimple and $\dim V \qmod G = \dim \fh = 1$ (hence $W
\simeq \Z/n\Z$) and to give some applications of the radial
component map in this situation. The function $k$ is then
given by $n-1$ complex parameters $k_1,\dots,k_{n-1}$, and
$R$, $U(k)$ are subalgebras of the fist Weyl algebra
$\C[z,\dz]$.  The paper is organized as follows.

In \S \ref{sec2} we recall general facts about Cherednik
algebras and their spherical subalgebras in the one
dimensional case. We show (Proposition~\ref{prop29}) that
$U(k) = \widetilde{U}/(\Omega)$ where $\widetilde{U}$ is an
algebra similar to $U(\fsl(2))$ (as defined in \cite{Sm})
and $\Omega$ is a generator of the centre of
$\widetilde{U}$. This says in particular that the
representation of $U(k)$ is well understood (and already
known).

In the third section we assume that $V$ is a representation
of the reductive group $\tG$, $G$ is the derived group of
$\tG$ and $\CVG = \C[f]$ for a non constant $f$. Then it is
known that: $\tG$ acts on $V$ with an open orbit,
i.e.~$\tGV$ is a prehomogeneous vector space (PHV), $S(V)^G=
\C[\Delta]$, $\Delta(f^{s+1}) =b(s) f^s$ where $b(s) =
c(s+1)(s+\alpha_1)\cdots(s+\alpha_{n-1})$ is the
(Bernstein-)Sato polynomial of $f$.  Choosing $k_i =
\alpha_i -1 + \frac{i}{n}$, $1 \le i \le n-1$, we prove that
$R= U(k)$ (Theorem~\ref{thm38}).

In section~\ref{sec4} we assume furthermore that the
representation $\tGV$ is multiplicity free (MF). By
\cite{HU} this is equivalent to the fact that $\DVtG =
\C[E_0,\dots,E_r]$ is a commutative polynomial ring. If
$\Theta$ is the Euler vector field on $V$ one can find
polynomials $b_{E_i}(s)$ such that, if $\Omega_i= E_i -
b_{E_i}(\Theta)$, $J= \sum_{i=0}^r \DVG \Omega_i$
(Theorem~\ref{thm410}). We then give a duality (of Howe
type) between representations of $G$ and lowest weight
modules over the Lie algebra generated by $f$ and $\Delta$
(which is infinite dimensional when $\deg f \ge 3$). This
duality recovers, and extends, results obtained by
H.~Rubenthaler \cite{Rub1} when $\tGV$ is of commutative
parabolic type.

In the last section we specialize further to the case where
$\tGV$ is of ``Capelli type'', i.e.~$\tGV$ is an irreducible
MF representation such that $\DVtG$ is equal to the image of
the centre of $U(\tfg)$ under the differential $\tau : \ftg
\to \DV$ of the $\tG$-action. These representations have
been studied in \cite{HU}, they fall into eight cases (see
Appendix~\ref{A1}). It is not difficult to see that $J=
\DVtauG$ when $\tGV$ is of Capelli type
(Proposition~\ref{prop53}) . We first apply this result to
study $\Dv$-modules of the form $\cM(g,k) =
\DV/\bigl(\DV\tau(\g) + \DV q(\Theta)Q_k\bigr)$ where $q(s)$
is a polynomial and $Q_k = f^k$ or $\Delta^k$. We show in
Theorem~\ref{thm57} that $\cM(g,k)$ is holonomic if and only
if $q(s) \ne 0$. This has the well known consequence that
the space of hyperfunction solutions of $\cM(g,k)$ is finite
dimensional. These properties generalize results obtained by
M.~Muro \cite{Mu4, Mu5}. For the second application, recall
first the classical fact \cite{Kac} that if $\tGV$ is MF,
there is a finite number of $\tG$-orbits $\cO_i$, $1 \le i
\le t$, in $V$. Let $\tilde{\calC}= \bigcup_{i=1}^t
\overline{T_{O_i}^*V}$ be the union of the conormal bundles
to the orbits. P.~Nang has shown that, when $\tGV= (\SO(n)
\times \C^* : \C^n)$, $(\GL(n) \times \SL(n) : \Mat_n(\C))$
or $(\GL(2n) : \Wedge^2 \C^{2n})$, the category
$\mod_{\tilde{\calC}}^{\mathrm{rh}}(\Dv)$ of regular
holonomic $\Dv$-modules whose characteristic variety is
contained in $\tilde{\calC}$ is equivalent to the category
$\mod^\theta(R)$ of finitely generated $R$-modules on which
$\theta= z\dz$ acts locally finitely. These representations
are of Capelli type. We conjecture (see
Conjecture~\ref{conj514}) that when $\tGV$ is of Capelli
type the category $\mod^{G \times \C^*}(\Dv)$ of $(G \times
\C^*)$-equivariant $\Dv$-modules is equivalent to
$\mod^\theta(R)$. If $G$ is simply connected, $\mod^{G
  \times \C^*}(\Dv) =
\mod_{\tilde{\calC}}^{\mathrm{rh}}(\Dv)$ and the conjecture
covers Nang's results; since $\mod^\theta(R)$ can be easily
described as a quiver category (i.e.~finite diagrams of
linear maps) its validity would give a simple classification
of $(G \times \C^*)$-equivariant $\Dv$-modules. One can
observe (Proposition~\ref{prop513}) that, as in \cite{Nang1,
  Nang3, Nang5}, the proof of the conjecture reduces to show
that any $M \in \mod^{G \times \C^*}(\Dv)$ is generated by
its $G$-fixed points.

%%%%%%%%%%%%%%%%%%%
\section{Rational Cherednik Algebras of Rank One}
\label{sec2}
%%%%%%%%%%%%%%%%%%%

\subsection{The spherical subalgebra and its restriction}
\label{ssec20}
In this section we summarize some of the results we will
need about rational Cherednik algebras. We begin with some
general facts, see for example \cite{EG, DO, Gor, Heck}.

Let $\h$ be a complex vector space of dimension $\ell$ and
$W \subset \GL(\h)$ be an arbitrary complex reflection
group.  Denote by $\calA = \{H_s\}_{s \in \calS}$ the
collection of reflecting hyperplanes associated to $W$
(where $s \in \calS \subset W$ is a complex reflection). Let
$\alpha_s \in \h^*$ such that $H_s= \alpha_s^{-1}(0)$ is the
reflecting hyperplane associated to $s \in \calS$.  Fix $H=
H_s \in \calA$; recall that the isotropy group $W_H = \{w
\in W : w_{\mid H}= \id_H\}$ is cyclic of order $n_H$ (this
order only depends on the conjugacy class of $s$). Let
$\esf_{H,i} \in \C W_H$, $0 \le i \le n_H-1$, be the
primitive idempotents of $\C W_H$. Fix a family
$$
k_{H_s,i} \in \C, \quad H_s \in \calA, \ \, 0 \le i \le
n_{H_s}-1, \ \, k_{H_s,0} = 0,
$$
of complex numbers such that $k_{H_s,i}= k_{H_t,i}$ if $s,t
\in \calS$ are conjugate. Such a family $k=(k_{H,i})_{H,i}$
is called a multiplicity function.  Let $\hreg$ be the
complement of $\bigcup_{s \in \calS} H_s$ and set $\pi =
\prod_{s \in \calS} \alpha_s$.  The group $W$ acts naturally
on $\Ch=S(\h^*)$, $\C[\hreg] = \Ch[\pi^{-1}]$, hence on
$\End_\C \Ch$ and $\End_\C \C[\hreg]$. These actions
restrict to $\Dh$ and $\cD(\hreg) = \Dh[\pi^{-1}]$. Denote
by $\cD(\hreg) \rtimes \CW$ the crossed product of the
algebra $\cD(\hreg)$ by the group $W$. Recall that in that
algebra we have: $wfw^{-1} = w.f$, $w\partial(y)w^{-1}
= \partial(w.y)$ if $f \in \Ch$ and $\partial(y)$ is the
vector field defined by $y \in V$.

% The operators $T_j$ and the algebra $\cH$ are defined, in
% coordinates, as follows. Let $\{e_1,\dots,e_\ell\}$ be
% basis of $\h$, with dual basis $\{x_1,\dots,x_\ell\}$ and
% denote by $\partial_i = \partial(e_i) =
% \frac{\partial}{\partial x_i}$ the corresponding
% derivatives. Then $\Ch = \C[x_i : 1 \le i \le \ell]$ and
% $$
% \Dh = \C[x_i, \partial_j : 1 \le i,j \le \ell] \subset
% \cD(\hreg) =\C[x_i, \partial_j : 1 \le i,j \le
% \ell][\pi^{-1}],
% $$

Then \cite{DO} one can introduce a subalgebra
$$
\cH = \cH(W,k) \subset \cD(\hreg) \rtimes \CW
$$
generated by three parts: $\Ch$, $W$, $\C[T(y) : y \in \h]
\cong \Sh$, where $T(y)$ is a Dunkl operator defined as
follows. Set $a_{H_s}(k) =
n_{H_s}\sum_{i=1}^{n_{H_s}-1}k_{H_s,i}\esf_{H_s,i} \in \C
W_{H_s}$ and
$$
T(y) = \partial(y) + \sum_{H_s \in \calA}
\dfrac{\dual{\alpha_s}{y}}{\alpha_s} a_{H_s}(k) \in
\cD(\hreg)\rtimes \CW.
$$
Denote by $\res : \cD(\hreg) \rtimes \CW \to
\End_\C\C[\hreg]$ the representation given by the natural
action of $W$ and $\cD(\hreg)$ on $\C[\hreg]$.  As observed
in \cite[\S 2.5]{DO} (see also \cite[Proposition~4.5]{EG})
$\res(\cH) \subset \End_\C\Ch$ and this gives a natural
structure of faithful $\cH$-module on $\Ch$, i.e.~we have an
injective homomorphism:
$$
\res : \cH \lto \End_\C \Ch.
$$
The group $W$ acts on $\cH$, $\cD(\hreg)\rtimes \CW$ and
$\End_\C \Ch$ by conjugation, i.e~$w.u = wuw^{-1}$. Denote
by $\cH^W \subset (\cD(\hreg)\rtimes \CW)^W$ and $(\End_\C
\Ch)^W$ the algebras of invariants under this action. Notice
that if $u \in \cH$, $w \in W$ and $f \in \Ch$, we have:
$\res(w.u)(f) = \res(wuw^{-1})(f)= \res(wu)(w^{-1}.f) =
w.\res(u)(w^{-1}.f) =(w.\res(u))(f)$.  Thus the homomorphism
$\res$ is $W$-equivariant and, in particular, $\res : \cH^W
\to (\End_\C \Ch)^W$.  Therefore $w.\res(u)(f) =
w.\res(u)(w^{-1}.f) = (w.\res(u))(f) = \res(u)(f) $, for all
$u \in \cH^W$, $w \in W$, $f \in \ChW$. We have obtained the
following representation of $\cH^W$ on $\ChW$:
$$
\res : \cH^W \lto \End_\C \ChW.
$$
(This morphism is not injective when $W \ne \{1\}$.)  Let
$$
\mathsf{e}= \frac{1}{|W|} \sum_{w \in W} w
$$
be the trivial idempotent and define the {\em spherical
  subalgebra}:
\begin{equation}
  \label{eq01}
  \Sph= \sph \subset \cH^W.
\end{equation}
Observe that $\Sph$ is an algebra whose unit is equal to
$\esf$. From the previous discussion we obtain $\Sph \subset
\esf(\cD(\hreg)\rtimes \CW)\esf$. It is not difficult to
show that $\esf(\cD(\hreg)\rtimes \CW)\esf =
\esf\cD(\hreg)^W \cong \cD(\hreg)^W$. It follows that $u\in
\sph$ can be written $u = \esf d$ for some $d \in
\cD(\hreg)^W$, hence $\res(u)(f) = d(f)$ for all $f \in
\ChW$. This implies that $\res(u) \in \End_\C \ChW$ acts as
the differential operator $d$ on $\ChW$. Consequently,
$\res(u) \in \cD(\h/W) = \cD(\ChW) \subset \cD(\hreg/W)
\cong \cD(\hreg)^W$.  Furthermore it is easy to see that
$d=0$ on $\ChW$ implies $d=0$, hence $u=0$. In conclusion:
one has the injective restriction morphism (see \cite{Heck}
in the case of a Weyl group):
\begin{equation}
  \label{eq02}
  \res: \sph \lto \DhmodW, \quad \forall \, f \in \ChW, \, D
  \in \cH^W, \
  \res(\mathsf{e}D)(f) = \res(D)(f).
\end{equation}
We set:
\begin{equation}
  \label{eq03}
  U= U(W,k) = \res(\sph) \subset \DhmodW.
\end{equation}

\subsection{The one dimensional case}
\label{ssec21}
We now go the most simplest case of the previous
construction: the case when $\ell = \dim \h =1$.

\begin{notation} Let $\h = \C v$ be a one dimensional vector
  space and $W \subset \GL(\h)$ be a finite subgroup of
  order $n$. We adopt the following notation.
  \begin{itemize}
  \item $\Ch=\Shs = \C[x]$, $\dual{x}{v}=1$,
    $\Dh=\C[x,\partial_x]$;
  \item $W = \langle w \rangle \simeq \Z/n\Z$, $w.x = \zeta
    x$ where $\zeta$ is a primitive $n$-th root of unity;
  \item $\ChW = \C[z]$, $z=x^n$, $\DhmodW=
    \C[z,\partial_z]$, $\theta = z\dz$;
  \item $\esf_0 = \esf, \esf_1,\dots, \esf_{n-1} \in \CW$
    are the primitive idempotents (hence $\esf_i =
    \frac{1}{n} \sum_{j=0}^{n-1} \zeta^{ij}w^j$);
  \item $k_0 =0,k_1,\dots,k_{n-1} \in \C$;
  \item $\displaystyle{T = T(v) = \partial_x +
      \frac{n}{x}\sum_{i=1}^{n-1} k_i \esf_i} \in \C[x^{\pm
      1},\partial_x] \rtimes \CW$;
  \item if $p(s) \in \C[s]$ is a polynomial, set: $\tau p(s)
    = p(s+1) - p(s)$, $\tau^{j+1}p(s) = \tau(\tau^{j}p)(s)$,
    $p^*(s) = p(s-1)$.
  \end{itemize}
\end{notation}

% We have denoted by $\C[x^{\pm 1},\partial_x] \rtimes \CW$
% the crossed product of the algebra $\C[x^{\pm
%   1},\partial_x]$ by the group $W$. Recall that in that
% algebra we have: $wx = \zeta w$, $w\dx = \zeta^{-1} \dx
% w$.

The following well known lemma will prove useful 
%%(we include a proof for completeness, 
(see~\cite{Kn1} for a more general statement).

\begin{lem}
  \label{lem21}
  Let $Q \in \C[z,\partial_z]$ satisfying:
  $$
  \exists \, p \in \Z, \quad \forall \, m \in \N, \quad
  Q(z^m) \in \C z^{m+p}.
  $$
  Then there exists a polynomial $\vphi(s) \in \C[s]$ of
  degree $d$ such that: $Q$ has order $d$ and can be written
  $$
  Q = z^p \vphi(\theta) = \sum_{j=0}^d q_j(z) \dz^{j}
  $$
  where
  $$
  q_j(z) = \frac{1}{j!}(\tau^{j}\vphi)(0) z^{j+p} \ \;
  \text{and $(\tau^{j}\vphi)(0) = 0$ if $p+j <0$.}
  $$
\end{lem}

\begin{rem}
  \label{rem22}
  One can define the algebra $\C[z^\alpha : \alpha \in \Q]$
  by adjoining roots of polynomials of the form $t^p - z$,
  $p \in \N$ prime. The derivation $\dz$ is naturally
  defined on this algebra by $\dz(z^\alpha) = \alpha
  z^{\alpha -1}$. Let $Q$ be as in Lemma~\ref{lem21}; then
  $Q$ extends to $\C[z^\alpha : \alpha \in \Q]$ by
  $Q(z^\alpha) = \sum_j q_j(z) \dz(z^\alpha) =
  \vphi(\alpha)z^{\alpha +p}$.
\end{rem}

The next lemma is straightforward by direct computation.

\begin{lem}
  \label{lem23}
  The following formulas hold:

  \noindent {\rm (a)} $[\esf_i,x] = x(\esf_{i+1} -
  \esf_{i})$ (where $\esf_n = \esf_0 = \esf)$;

  \noindent {\rm (b)} $[T,x] = 1 + n
  \sum_{i=1}^{n-1}k_i(\esf_{i+1} - \esf_i) = 1 + n
  \sum_{i=0}^{n-1}(k_i - k_{i+1}) \esf_i$;

  \noindent {\rm (c)} $wTw^{-1} = \zeta^{-1}T$;

  \noindent {\rm (d)} let $p \in \N$ and define $q \in
  \{0,\dots,n-1\}$ by $p + q \equiv 0 \pmod{n}$, then
  $T(x^p) = (n k_q + p)x^{p-1}$;

  \noindent {\rm (e)} let $1 \le j \le n$ and $s \in \N$,
  then
  $$
  T^j(x^{sn}) = \prod_{i=1}^j(nk_{i-1} + sn - i +1) x^{sn
    -j},
  $$
  in particular $(T/n)^n(z^s) = \prod_{i=0}^{n-1}(s + k_i -
  i/n) z^{s -1}$.
\end{lem}

We now introduce the rational Cherednik algebra, and its
spherical subalgebra, in the rank one case.

\begin{defn}
  \label{def24}
  The rational Cherednik algebra associated to $W$ with
  parameters $k_i$, $0 \le i \le n-1$, is the subalgebra of
  $\C[x^{\pm 1},\partial_x] \rtimes \CW$ defined by:
  $$
  \cH= \cH(W,k_0,\dots, k_{n-1})= \C \langle x ,T, w \rangle
  $$
  Its spherical subalgebra is $\Sph$.
\end{defn}

Observe that when $n=1$ (i.e.~$W$ trivial) the algebra $\cH
= \Sph$ is nothing but $\Dh=\C[x,\dx]$ and all the results
we are going to obtain are in this case obvious. {\em We
  therefore will only be interested in the case $n \ge 2$}.

It is easily seen that:
\begin{itemize}
\item $\Sph = \sph= \C \langle \esf, \esf x^n, \esf (T/n)^n,
  \esf xT/n \rangle$;
\item the image $U = \res(\Sph)$ of the injective
  homomorphism, defined in~\eqref{eq02},
  \begin{equation}
    \label{eq25}
    \res : \sph \lto \DhmodW = \C[z,\partial_z], 
    % \quad \res(\esf D)(f) = D(f) \ \; \text{for all $f \in
    %   \ChW$}
  \end{equation}
  is generated by $z$, $\res\bigl(\esf (T/n)^n\bigr)$ and
  $\res(\esf xT/n)$.
\item there exists a finite dimensional filtration on $\Sph$
  such that the associated graded algebra $\gr(\Sph)$ is
  isomorphic to $S(\h^* \times \h)^W \equiv \C[X,Y,S]/(XY -
  S^n)$, cf.~\cite[p.~262]{EG} (one has $X \equiv \gr(\esf
  x^n), Y \equiv \gr(\esf T^n), S \equiv \gr(\esf xT)$).
\end{itemize}

Fix a constant $c \in \C^*$ and set:
\begin{gather}
  \lambda_i = k_i -\frac{i}{n}, \ \; b^*(s) = c
  \prod_{i=0}^{n-1}(s + \lambda_i), \ \; b(s) = b^*(s+1) = c
  \prod_{i=0}^{n-1}(s + \lambda_i + 1) \label{eq26}
  \\
  v(s) = -2b(-s), \quad \psi(s) = \half(\tau v)(s) = b(-s) -
  b(-s-1) \label{eq27}.
\end{gather}

\begin{prop}
  \label{prop28}
  Set $\delta = c \res\bigl(\esf (T/n)^n\bigr)$. Then $U =
  \res(\Sph) = \C[z,\theta,\delta]$ and one has:
  
  {\rm (1)} $\delta = z^{-1}b^*(\theta) = \sum_{j=1}^n
  \frac{1}{j!}(\tau^jb^*)(0) z^{j-1}\dz^j$;
  
  {\rm (2)} $\res(\esf xT/n) = \theta$;
  
  {\rm (3)} $[\delta,z]= \psi(-\theta) = b(\theta) -
  b(\theta-1) = (\tau b^*)(\theta)$;
  
  {\rm (4)} $[\theta,z] = z$, $[\theta,\delta] = -\delta$;
  
  {\rm (5)} $2z\delta + v(-\theta + 1) = 2\bigl(z\delta -
  b^*(\theta)\bigr) = 0$.
\end{prop}

\begin{proof}
  The equality $U = \C[z,\theta,\delta]$ is clear.
  
  (1) From the definition of the map $\res$,
  cf.~\eqref{eq02}, and Lemma~\ref{lem23}(e) we deduce that
  $\delta(z^s) = c(T/n)^n(z^s) = b^*(s) z^{s-1}$. The claim
  therefore follows from Lemma~\ref{lem21} applied to $Q=
  \delta$, $\vphi = b^*$ and $p=-1$.
  
  (2) By Lemma~\ref{lem23}(e) again we get that
  $\res(xT)(z^s) = xT(x^{sn}) = sn x^{sn}= n sz^s$,
  i.e.~$\res(\esf xT/n) = \theta$.
  
  (3) Using (1) we obtain that $[\delta,z](z^s) = (\tau
  b^*)(s)z^s = (\tau b^*)(\theta)(z^s)$, hence $[\delta,z] =
  (\tau b^*)(\theta)$.
  
  The formulas in (4) and (5) are obvious.
\end{proof}

\subsection{Algebras similar to $U(\fsl(2))$}
\label{ssec22}
We recall here the definition, and some properties, of the
algebras similar to $U(\fsl(2))$ introduced in \cite{Jo} and
\cite{Sm}.

Let $\psi(s) \in \C[s]$ be an arbitrary polynomial of degree 
$\ge 1$ and write $\psi = \half\tau v$ for some $v \in
\C[s]$ of degree $n \ge 2$. Define a $\C$-algebra
$\widetilde{U}$ by generators and relations as follows
(cf.~\cite{Sm}):
$$
\widetilde{U}= \widetilde{U}(\psi) =C \langle A,B,H \rangle,
\quad [A,B] -\psi(H) = 0, \quad [H,A] - A = 0, \quad [H,B] +
B = 0.
$$
Note that when $\deg \psi = 1$, i.e.~$n=2$, one has
$\widetilde{U} = U(\fsl(2))$.  The algebra $\widetilde{U}$
has the following properties, see \cite{Sm,JvO,MVdB}.
\begin{itemize}
\item The centre of $\widetilde{U}$ is $Z(\widetilde{U}) =
  \C[\Omega]$, $\Omega = 2BA + v(H+1) = 2AB + v(H)$.
\item For $\lambda \in \C$ one defines the ``Verma module''
  $M(\lambda) = \widetilde{U} \otimes_{\C[H,A]} \C_\lambda$,
  where $\C_\lambda$ is the one dimensional module
  associated to $\lambda$ over the solvable Lie algebra $\C
  A + \C H$.
\item Each $M(\lambda)$ has a unique simple quotient
  $L(\lambda)$ and any finite dimensional
  $\widetilde{U}$-module is of the form $L(\lambda)$ for
  some $\lambda$.
\item The primitive ideals of $\widetilde{U}$ are the
  annihilators $\ann L(\lambda)$; the minimal primitive
  ideals are the $\ann M(\lambda) = (\Omega -
  v(\lambda+1))$, they are completely prime. If $I$ is an
  ideal strictly containing $\ann M(\lambda)$, then $\dim_\C
  \widetilde{U}/I$ is finite.
\item One can define in an obvious way a category $\cO$ for
  $\widetilde{U}$ which decomposes as:
  $$
  \cO = \bigsqcup_\alpha \cO_\alpha, \quad \cO_\alpha= \{ M
  \in \cO : (\Omega - \alpha)^k M = 0 \ \text{for some
    $k$}\}.
  $$
  Moreover $\cO_\alpha \equiv \mathrm{mod} A$ for a finite
  dimensional $\C$-algebra $A$.
\end{itemize}
The representation theory of the algebras
$\widetilde{U}/(\Omega - v(\lambda+1))$ is therefore quite
well understood.

We will be interested in the algebra $U =
\widetilde{U}/(\Omega) = \C[a,b,h]$ where $a,b,h$ are the
classes of $A,B,H$. We have in $U$:
$$
[a,b] = \psi(h) = \half (\tau v)(h), \quad [h,a] = a, \quad
[h,b] = -b, \quad 2ab = -v(h).
$$
For simplicity we will assume that $v(1) =0$. Recall then
that $M(0) \equiv \C[b]$ is a faithful $U$-module where
$h.b^k= -k b^k$, $b.b^k = b^{k+1}$, $a.b^k =
\bigl(\sum_{i=1}^{k-1}\psi(-i)\bigr)b^{k-1}$ for all $k \ge
0$. We want to study the Lie subalgebra $\calL$ of
$(U,[\phantom{x},\phantom{x}])$ generated by the elements
$a,b$. Recall \cite{Sm} that when $n \le 2$ this algebra is
finite dimensional. The algebra $\calL$ acts on $\C[b]\equiv
M(0)$ and for each $i \in \Z$ we set:
$$
\calL_i = \{u \in \calL : u.b^k \in \C b^{k+i} \ \text{for
  all $k \in \N\}$}.
$$
Clearly: $a \in \calL_{-1}$, $h \in \calL_0$, $b \in
\calL_1$.

\begin{lem}
  \label{lem281} {\rm (1)} The element $h$ is transcendental
  over $\C$.

  \noindent {\rm (2)} For $g(h) \in \C[h]$ we have
  $[b,ag(h)] = -\half \tau(vg^*)(h)$.

  \noindent {\rm (3)} There exists a sequence $(g_m(h))_{m
    \in \N} \subset \calL \cap \C[h]$ such that $\deg g_m =
  (m+1)n -(2m+1)$. In particular, $\deg g_m < \deg g_{m+1}$
  when $n \ge 3$.
\end{lem}

\begin{proof}
  (1) This follows, for example, from $g(h).b^k = g(-k)b^k$
  in $M(0)$ for all $g(h) \in \C[h]\subset U$.

  \noindent (2) Recall \cite[Appendix]{Sm} that $[b,g(h)] =
  b(g(h) - g(h-1))$ and $[g(h),a] = a \tau(g)(h)$. Thus:
  \begin{align*} [b,ag(h)] &= [b,a]g(h) + a[b,g(h)] = -\half
    (\tau v)(h) g(h)
    +ab (g(h) - g(h-1)) \\
    &= -\half(v(h+1) - v(h))g(h) - \half v(h)(g(h) - g(h-1))\\
    & = -\half(v(h+1)g(h) - v(h)g(h-1)) \\
    &= -\half (\tau (vg^*))(h),
  \end{align*}
  as desired.
  
  (3) We start with $g_0(h) = \psi(h) = [a,b] \in \calL$.
  Then, $\deg g_0 = n - 1$ and, by (2), $[b,a(\tau g_0)(h)]
  = -\half \tau(v(\tau g_0)^*)(h) \in \calL$. We thus set
  $g_1(h)= -\half \tau(v(\tau g_0)^*)(h)$; note that $\deg
  \tau(v(\tau g_0)^*) = \deg v + \deg \tau g_0 - 1 = 2n -3$.
  Suppose that $g_m(h) \in \calL$ has been obtained; then
  $[g_m(h),a] = a (\tau g_m)(h) \in \calL$ and
  $[b,[g_m(h),a]] = -\half \tau(v(\tau g_m)^*)(h) \in
  \calL$. Set $g_{m+1}(h) = \tau(v(\tau g_m)^*)(h)$. We have
  $\deg g_{m+1} = \deg v + \deg g_m -2 = (m+1)n -(2m+1) + n
  -2 = (m+2)n - (2(m+1) + 1)$, hence the result.
\end{proof}

\begin{prop}
  \label{prop282}
  Assume that $n \ge 3$. Then $\dim_\C \calL_i = \infty$ for
  all $i \in \Z$.
\end{prop}

\begin{proof}
  When $i=0$ the claim follows from Lemma~\ref{lem281}(1).
  Suppose that $i >0$. Let $(g_m(h))_m \subset \calL_0$ be
  as in Lemma~\ref{lem281}(3).  We will now show that
  $a^i(\tau^ig_m)(h) \in \calL_{-i}$ for all $m$. Since
  $\deg g_{m+1} > \deg g_m$ (because $n \ge 3$) the elements
  $a^i(\tau^ig_m)(h)$ are linearly independent in the domain
  $U$.  We argue by induction on $i$.  When $i=1$ the claim
  follows from $[a,g_m(h)] = a(\tau g_m)(h) \in \calL_{-1}$
  for all $m$. Assume that $a^i(\tau^ig_m)(h) \in
  \calL_{-i}$ for all $m$, then: $[a,a^i(\tau^ig_m)(h)] =
  a^{i+1}(\tau^{i+1} g_m)(h) \in \calL_{-(i+1)}$ for all
  $m$.
  
  Observe that there exists an anti-automorphism of
  $\widetilde{U}$ given by $\varkappa(A) = B$, $\varkappa(B)
  = A$, $\varkappa(H)= H$. It satisfies $\varkappa(\Omega) =
  \Omega$, therefore $\varkappa$ induces an
  anti-automorphism of ${U}$. Since
  $\varkappa(a^{i}(\tau^{i} g_m)(h))= (\tau^{i} g_m)(h)b^{i}
  \in \calL_i$, we get that $\dim \calL_i = \infty$.
\end{proof}

\begin{remark}
  The fact that $\dim_\C \calL = \infty$ when $n \ge 3$ can
  also be proved by using \cite{RST}.
\end{remark}

The next result shows that the spherical subalgebra $\Sph$
is isomorphic to a quotient $\widetilde{U}/(\Omega)$ for an
an obvious choice of $\psi(s)$:

\begin{prop}
  \label{prop29}
  Let $b^*(s), v(s) = -2b(-s), \psi(s) = \half (\tau v)(s)
  \in \C[s]$ be as in~\eqref{eq26} and~\eqref{eq27}, and
  denote by $U = \res(\Sph) = \C[z,\theta,\delta]$ the image
  of the spherical subalgebra under the restriction map. Let
  $\widetilde{U} = \C \langle A,B,H \rangle$ be the algebra
  similar to $U(\fsl(2))$ defined by $\psi$. Then the
  morphism
  $$
  \pi : \widetilde{U} \lto U, \quad \pi(A) = \delta, \
  \pi(B) =z, \ \pi(H) = h = -\theta
  $$
  induces an isomorphism $\pi : \widetilde{U}/(\Omega)
  \longisomto U$. In particular, there exists an isomorphism
  $\widetilde{U}/(\Omega) \isomto \Sph$.
\end{prop}

\begin{proof}
  The existence of the surjective homomorphism $\pi$ clearly
  follows from (3) and (4) in Proposition~\ref{prop28}. By
  Proposition~\ref{prop28}(5), $\Ker(\pi)$ contains
  $(\Omega)$.  Since $\pi(\widetilde{U})= U$ is not finite
  dimensional, it follows from one of the properties of
  $\widetilde{U}$ that $\Ker(\pi)= (\Omega)$.
\end{proof}

From results of Smith~\cite{Sm}, Mus\-son--Van den
Bergh~\cite{MVdB}, et al., one can for instance deduce the
following properties of the the algebra $U\cong \Sph$:
\\
-- The Verma modules over $U$ are the $M(\lambda)$'s such
that $v(\lambda + 1) = -2b^*(-\lambda) = 0$, i.e.~$\lambda =
0,\lambda_1,\dots,\lambda_{n-1}$, cf.~\eqref{eq26}.
\\
-- $\dim L(\lambda) < \infty$ $\iff$ $b^*(-\lambda) =
b^*(-\lambda + j) = 0$ for some $j \in \N^*$.
\\
-- The global homological dimension of $U$ is:
$$
% \gldim U =
\begin{cases}
  \infty  \  & \text{if $b^*(-s)$ has multiple roots;} \\
  2 \ & \text{if $b^*(-s)$ has no multiple root and two
    roots differing by some
    $j \in \N^*$;} \\
  1 \ & \text{otherwise.}
\end{cases}
$$

When a root $-\lambda$ of the polynomial $b^*(s)$ is a
rational number one can use Remark~\ref{rem22} to realize
the Verma module $M(\lambda)$ under the form
$\C[z]z^{-\lambda}$, on which $z,\theta,\delta$ act as
differential operators. This is for example the case for
$\lambda = 0$, where we have $M(0) = \C[z]$.

\begin{examples}
  \noindent (1) Take $b(s) = (s+1)(s+3/2)
  \cdots(s+(n+1)/2)$, $\lambda_i = i/2$, $0 \le i \le
  (n-1)/2$.  Verma modules: $M(\lambda_i) = \C[z] z^{-i/2}$
  with the natural action of $z,\delta,\theta$.  Irreducible
  finite dimensional modules: $L(\lambda_i) =
  M(\lambda_i)/M(\lambda_i-1)$, $i \ge 2$, of
  dimension~$1$.  The global dimension of $U$ is $2$.
  
  \noindent (2) Take $b(s) = (s+1)(s+5)(s+9)$, $\lambda_i =
  4i$, $i=0,1,2$.  Verma modules: $M(4i) = \C[z] z^{-4i}$,
  $M(0) \subset M(4) \subset M(8)$ with quotients $L(4i)
  = M(4i)/M(4(i-1))$ simple of dimension $4$.  The global
  dimension of $U$ is $2$.

  \noindent (3) Take $b(s) = (s+1)(s+n/2)$, $n \ge 3$,
  $\lambda_0 = 0$, $\lambda_1 = \frac{n-2}{2}$.  Verma
  modules: $M(0) = \C[z]$, $M(\frac{n-2}{2}) =
  \C[z]z^{-\frac{n-2}{2}}$.  There are two cases:
  \begin{itemize}
  \item $n=2k$, then $M(0) = \C[z] \subset M(k-1) =
    \C[z]z^{-(k-1)}$ with quotient $L(k-1)$ simple of
    dimension $k-1$; the global dimension of $U$ is $2$;
  \item $n=2k+1$, then $M(0) = \C[z]$ and $M(k-\half) =
    \C[z]z^{-(k-\half)}$ are simple; the global dimension of
    $U$ is $1$.
  \end{itemize}

  \noindent (4) Take
  \begin{equation*}
    \begin{split}
      b(s) = (s+1)^8 &
      \left[(s+2/3)(s+4/3)(s+3/4)(s+5/4)(s+5/6)(s+7/6)\right]^4 \\
      & \times
      \left[(s+7/10)(s+9/10)(s+11/10)(s+13/10)\right]^2,
    \end{split}
  \end{equation*}
  $\lambda_0= 0$, $\lambda_1= -1/3$, $\lambda_2= 1/3$,
  $\lambda_3= -1/4$, $\lambda_4= 1/4$, $\lambda_5=-1/6$,
  $\lambda_6 = 1/6$, $\lambda_7 = -3/10$, $\lambda_8 =
  -1/10$, $\lambda_9 = 1/10$, $\lambda_{10}=3/10$.  Verma
  modules: $M(\lambda_i) = \C[z]z^{-\lambda_i}$ which are
  simple $U$-modules.  The global dimension of $U$ is
  $\infty$.
\end{examples}

%%%%%%%%%%%%%%%%%%%%
\section{Representations with a one dimensional quotient}
\label{sec3}
%%%%%%%%%%%%%%%%%%%%

\subsection{Prehomogeneous vector spaces}
\label{ssec31}
Let $\tG$ be a connected reductive complex algebraic group.
We denote by $G = (\tG,\tG)$ its derived subgroup, which is
a connected semisimple group. Recall that $\tG = GC$ where
$C = Z(\tG)^0$, the connected component of the center of
$\tG$, is a a torus.
% , of dimension $a$.

Let $\tilde{\rho} : \tG \to \GL(V)$ be a finite dimensional
representation of $\tG$. Recall that $f \in \CV=S(V^*)$ is a
relative invariant of $\tGV$ if there exists a rational
character $\chi \in \mathsf{X}(\tG)$ such that $g.f =
\chi(g)f$ for all $g \in \tG$.

One says, see~\cite[Chapter~2]{Kim}, that $\tGV =
(\tG,\tilde{\rho},V)$ is a (reductive) {\em prehomogeneous
  vector space} (PHV) if $\tilde{G}$ has a dense orbit in
$V$. We denote the complement of the dense orbit by $S$, it
is called the singular set of $\tGV$.  Then it is known
\cite[Theorem~2.9]{Kim} that the one-codimensional
irreducible components of $V \smallsetminus S$ are of the
form $\{f_i = 0\}$, $1 \le i \le r$, for some relative
invariants $f_i$. The $f_i$ are algebraically independent
and are called the basic relative invariants of $\tGV$; any
relative invariant $f$ can be (up to a non-zero constant)
written as $\prod_{i=1}^r f_i^{m_i}$. When the singular set
is a hypersurface $\tGV$ is called regular,
cf.~\cite[Theorem~2.28]{Kim}.
% Then $S= \{ f =0\}$ for some relative invariant $f$.

Let $\tilde{\rho}^* : \tG \to \GL(V^*)$ be the
contragredient representation. Then, see
\cite[Proposition~2.21]{Kim}, $\tGVd$ is a PHV.  Recall that
$S(V)=\C[V^*]$ can be identified with the algebra of
constant coefficients differential operators on $V$. If
$\vphi \in \C[V^*]$ we denote by $\vphi(\partial)$ the
corresponding differential operator. If $f \in \CV$ is a
relative invariant of degree $n$ and weight $\chi \in
\Xt(\tG)$, there exists a relative invariant $f^* \in
\C[V^*]$
% of $\tGVd$,
of degree $n$ and weight $\chi^{-1}$. The following result
summarizes \cite[Proposition~2.22]{Kim} and \cite{Ka2}.

\begin{thm}[Sato-Bernstein-Kashiwara]
  \label{thm31}
  Under the above notation, set $\Delta= f^*(\partial) \in
  S(V)$.  There exists $b(s) \in \R[s]$ of degree $n$ such
  that:
  \begin{enumerate}
  \item[{\rm (1)}] $b(s) = c \prod_{i=0}^{n-1}(s+\lambda_i
    +1)$, $c > 0$;
  \item[{\rm (2)}] $\Delta(f^{s+1}) = b(s) f^s$;
  \item[{\rm (3)}] $\lambda_i + 1 \in \Q_+^*$, $0 \le i \le
    n-1$, $\lambda_0 = 0$.
  \end{enumerate}
\end{thm}

The polynomial $b(s)$ is called a \textit{$b$-function} of
$f$.  Since the form of the operator $\Delta=f^*(\partial)$
will be important in the proof of Theorem~\ref{thm38}, we
briefly indicate its expression in a particular coordinate
system (see \cite[p.~38]{Kim}).

Denote by $\overline{a}$ the complex conjugate of $a \in \C$
and set $|a|^2 = a \overline{a} \in \R_+$.  Let $\tK$ be a
maximal compact subgroup of $\tG$, so that
$\tG=\tK\exp(i\tilde{\fk})$ is the complexification of $\tK$
(where $\tilde{\fk} = \Lie(\tK)$).  Fix a $\tK$-invariant
non-degenerate Hermitian form $\kappa$ on $V$ such that
$\kappa(\lambda v, \mu w) =
\overline{\lambda}\mu\kappa(v,w)$, $\lambda, \mu \in \C$,
$v,w \in V$. We choose a $\kappa$-orthonormal basis
$\{e_i\}_{1 \le i \le N}$ on $V$, with dual basis $\{z_i =
e_i^*\}_i$. Define
\begin{equation}
  \label{eq31}
  \phi :  V \lto V^*, \quad \phi(v)= v^*= \kappa(v,-).
\end{equation}
In coordinates we have: $\phi\bigl(\sum_i v_i e_i\bigr)=
\sum_i\overline{v_i}z_i$. Then $\phi$ is a bijective
$\C$-antilinear map such that $\phi(h.v)= h.\phi(v)$ for all
$h \in \tK$. The inverse of $\phi$, also denoted by $\phi$,
is given by $\phi(v^*) = v$, i.e.~$\phi\bigl(\sum_i a_i
z_i\bigr)= \sum_i\overline{a_i}e_i$, and it also satisfies
$\phi(h.v^*) = h.\phi(v^*)$.  We can extend $\phi$ to
$S(V)=\C[V^*]$ and $S(V^*)=\CV$ by multilinearity. Thus we
get a bijective $\tK$-equivariant $\C$-antilinear morphism:
\begin{equation}
  \label{eq32}
  \phi : \CV \lto \C[V^*], \quad \phi(f) = f^*.
\end{equation}
Now, if $f$ is a relative invariant of $\tGV$ associated to
$\chi$, we obtain:
$$
h.f^*= h.\phi(f)= \phi(h.f) = \phi(\chi(h)f) =
\overline{\chi(h)} \phi(f) = \chi(h)^{-1}f^*
$$
for all $h \in \tK$, which shows that $f^*$ is a relative
invariant corresponding to $\chi^{-1}$.

The expression of $\phi(f)=f^*$ in the chosen basis is given
as follows. If $\mathtt{i}= (i_1,\dots,i_N) \in \N^N$ we set
$|\mathtt{i}| = \sum_{j=1}^N i_j$, $\mathtt{i!}=
\prod_{j=1}^N i_j$, $z^{\mathtt{i}}= z_1^{i_1}\cdots
z_N^{i_N}$, $e^{\mathtt{i}}= e_1^{i_1}\cdots e_N^{i_N} $ and
$\partial^{\mathtt{i}}= \partial_{z_1}^{i_1}\cdots
\partial_{z_N}^{i_N}$.  Then, if the polynomial $f$ is
written
\begin{equation}
  \label{eq33}
  f= \sum_{\{\mathtt{i} \in \N^N \, : \, |\mathtt{i}| = n\}}
  a_{\mathtt{i}} z^{\mathtt{i}}, \quad a_{\mathtt{i}} \in
  \C, 
\end{equation}
one has
\begin{equation}
  \label{eq34}
  f^* = \sum_{\{\mathtt{i} \in \N^N \, : \, |\mathtt{i}| =
    n\}} \overline{a_{\mathtt{i}}} 
  \, e^{\mathtt{i}}
\end{equation}
and therefore:
\begin{equation}
  \label{eq35}
  \Delta =\sum_{\{\mathtt{i} \in \N^N \, : \, |\mathtt{i}| =
    n\}} \overline{a_{\mathtt{i}}} 
  \, \partial^{\mathtt{i}}, \quad b(0) = \Delta(f) = c
  \prod_{i=0}^{n-1}(\lambda_i +1) = 
  \sum_{\mathtt{i}\in \N^N} \mathtt{i!}
  |a_{\mathtt{i}}|^2.
\end{equation}

\begin{rems}
  \label{rem32}
  1) If $K$ is a maximal compact subgroup of the semisimple
  group $G$ we can embed $K$ in a maximal compact subgroup
  $\tK$ of $\tG$.  Note that any relative invariant of
  $\tGV$ is a $G$-invariant, and that $f$ is $G$-invariant
  if and only if it is $K$-invariant.  The previous
  construction shows that there exist $\tK$-equivariant
  bijective $\C$-antilinear morphisms $\phi : S(V) \to
  S(V^*)$ and $\phi : S(V^*) \to S(V)$ such that $\phi \circ
  \phi = \id$. In particular we have
  $\phi\bigl(S(V^*)^K\bigr) = S(V)^K$, hence
  $\phi\bigl(S(V^*)^G\bigr) = S(V)^G$.
  
  2) Observe that $\phi : V \to V^*$ is $K$-equivariant,
  but, in general, there is no $G$-module isomorphism of
  between $V$ and $V^*$. For example, in the case $\tGV =
  \bigl(\GL(n,\C\bigr), \Wedge^2\C^n)$ one has
  ($\Wedge^2\C^n)^* \cong \Wedge^{n-2}\C^n$ as a $G$-module,
  which is not isomorphic to $\Wedge^2\C^n$ when $n > 2$.
\end{rems}

\subsection{PHV of rank one}
\label{ssec32}
Let $\tG = GC$ be as above and $\tGV$ be a finite
dimensional representation of $\tG$. In this subsection we
make the following hypothesis:

\begin{mainhyp}
  \label{hypA}
  There exists $f \in S^n(V^*)$ such that: %$n \ge 1$,
  $f \notin \CV^{\tG}$ and $\CV^G= \C[f]$.
\end{mainhyp}

\begin{remarks}
  (1) Assume that $G$ is a semisimple group and $\GV$ is a
  finite dimensional representation of $G$ such that $\CV^G=
  \C[f]$, $f \notin \C^*$. Let the group $\C^*$ act on $V$
  by homotheties. Then $(G\times \C^* : V)$ satisfies the
  hypothesis~\ref{hypA}. Therefore one could assume without
  lost of generality that $C=\C^*$.
  
  (2) Let $f$ be as in the previous hypothesis. Then, since
  $G$ is semisimple, the polynomial $f$ is irreducible. Let
  $g \in C$, then $g.f \in \CV^G= \C[f]$ and $\deg(g.f) =
  \deg(f)$, hence $g.f = \chi(g)f$ for some $\chi(g) \in
  \C^*$. It follows that $\chi \in \mathsf{X}(\tG)$,
  i.e.~$f$ is a relative invariant of $\tGV$; note that
  $\chi \ne 1$, since $f \notin \CV^{\tG}$.  Assume that
  $f_1$ is another relative invariant of $\tGV$.  Then $f_1
  \in \CV^G$ is homogeneous and this implies that $f_1=
  \alpha f^m$ for some $\alpha \in \C, m \in N$.
\end{remarks}

\begin{prop}
  \label{prop32}
  Let $f$ be as in hypothesis~\ref{hypA}.

  \noindent {\rm (i)} Let $\C(V)$ be the fraction field of
  $\CV$. Then one has:
  $$
  \C(V)^G= \C(f), \qquad \C(V)^\tG= \C.
  $$
  In particular, $\tGV$ is a PHV.

  \noindent {\rm (ii)} Let $f^*$ be the relative invariant
  obtained in~\eqref{eq33}, then $\C[V^*]^G = \C[f^*]$.

  \noindent {\rm (iii)} The representation $\GV$ is polar.
\end{prop}

\begin{proof}
  (i) The equality $\C(V)^G= \C(f)$ follows from
  \cite[Theorem~3.3]{PV}. By the remark~(2) above, we know
  that $g.f= \chi(g)f$ for all $g \in \C$. Observe that
  $\C(V)^\tG = (\C(V)^G)^C=\C(f)^C$ and let $\vphi
  =p(f)/q(f) \in \C(f)^C$ where $p(f), q(f) \in \C[f]$ are
  relatively prime polynomials in $f$. One easily sees that
  $p(f)$ and $q(f)$ are relative invariants, thus $p(f) =
  \alpha f^k$ and $q(f) = \beta f^{\ell}$, $\alpha, \beta
  \in \C$. It follows that $\chi^{k-\ell} =1$, hence
  $k=\ell$ and $\vphi \in \C$.  From \cite[Corollary,
  p.~156]{PV} one gets that $\tGV$ is a PHV.

  \noindent (ii) Adopt the notation of
  Remarks~\ref{rem32}. The map $\phi : \CV \to \C[V^*]$
  defined in~\eqref{eq32} yields a bijective $\C$-antilinear
  morphism $\C[V]^G= \C[V]^K \to \C[V^*]^K = \C[V^*]^G$.
  Thus $\C[V^*]^G = \C[f^*]$.

  \noindent (iii) Choose $v \in V$ regular semisimple,
  i.e.~$G.v$ closed and $\dim G.v \ge \dim G.v'$ for all
  closed orbits $G.v'$. Set $\h = \C v$ and $\g= \Lie(G)$.
  Then, see~\cite{DK}, one easily deduces that $\h= \h_v =
  \{ x \in V : \g.x \subset \g.v\}$ is a Cartan subspace for
  the $G$-action on $V$.
\end{proof}

{\em From now on, we assume that the hypothesis~\ref{hypA}
  is satisfied} and we fix a Cartan subspace $\h = \C v$ for
the $G$-action on $V$.  We set
\begin{itemize}
\item $x= v^*$, hence $\C[\h] = S(\h^*) = \C[x]$;
\item $W = N_G(\h)/Z_G(\h)$.
\end{itemize}
By~\cite{DK} we know that there exists an isomorphism $V
\qmod G \cong \h/W$ given by the restriction map $\psi: \CVG
\longisomto \ChW$, $\psi(p(f)) = p(f)_{\mid \h}$. Since $f$
is homogeneous of degree $n$, $\psi(f) \in \C[x]$ is a
scalar multiple of $z=x^n$. Therefore, multiplying $v$ by a
non zero constant, we may assume that
$$
\psi(f) = x^n, \qquad W \equiv \langle w \rangle \subset
\GL(\h)
$$
where $w$ acts on $x$ by $w.x = \zeta x$, $\zeta$ primitive
$n$-th root of unity. We can therefore adopt the notation of
\S\ref{ssec21}. In particular, $\ChW= \C[z]$ where $z = x^n
= \psi(f)$.

Let $b(s) = c \prod_{i=0}^{n-1}(s+\lambda_i +1) \in \C[s]$
be a $b$-function of the relative invariant $f$ as in
Theorem~\ref{thm31}. We can then define the rational
Cherednik algebra
$$
\cH = \cH(W,k_0,\dots,k_{n-1})= \C\langle x,T,w\rangle,
$$
where the parameters $k_i$ are given by $k_i = \lambda_i +
\frac{i}{n}$, cf.~\eqref{eq26}.  Recall that the image $U=
\res(\Sph)\subset \C[z,\dz]$ of the spherical subalgebra is
generated by $z,\delta,\theta$, where
$$
\theta = z \dz, \qquad \delta = z^{-1}b^*(\theta) =
z^{-1}c\prod_{j=0}^{n-1}(\theta + \lambda_j),
$$
see Proposition~\ref{prop28}.

Recall from \S\ref{sec1} that we have a radial component
map:
$$
\rad : \DVG \lto \DhmodW = \C[z,\dz], \quad \rad(D)(p(z)) =
\psi(D(p(f))),
$$
for all $p(z) \in \ChW$. The {\em algebra of radial
  components} is defined to be
\begin{equation}
  \label{radial}
  R= \rad(\DVG) \subset \C[z,\dz].
\end{equation}
The aim of this section is to prove that $R = U$, see
Theorem~\ref{thm38}.

Before entering the proof, let us give some standard
examples. A complete list of pairs $\tGV$ as above with $V$
irreducible can be found in \cite{Kim}. Recall that $\Delta
= f^*(\partial) \in S(V)$ is the differential operator
constructed in~\eqref{eq35}.

\begin{examples}
  (1) $(\tG = \GL(n) : V = S^2\C^n)$, $W = \Z_n$, $f =
  \det(x_{ij})$, $\Delta= \det(\partial_{x_{ij}})$, $b(s) =
  \prod_{i=0}^{n-1}(s + i/2 + 1)$.

  \noindent (2) $(\tG = \mathrm{E}_6 \times \C^* : V =
  \C^{27})$, $W = \Z_3$, $f=$ cubic form, $b(s) =
  (s+1)(s+5)(s+9)$.
    
  \noindent (3) $(\tG = \SO(n) \times \C^* : V =\C^n)$, $W =
  \Z_2$, $f$ = quadratic form, $\Delta$ = Laplacian, $b(s) =
  (s+1)(s+n/2)$.
  
  \noindent (4) $(\tG = \SL(5) \times \GL(4) : V =
  \Wedge^2\C^5 \otimes\C^4)$, $W= \Z_{40}$, $\deg f = 40$,
  $b(s)$ is:
  \begin{align*}
    (s+1)^8 &
    \bigl[(s+2/3)(s+4/3)(s+3/4)(s+5/4)(s+5/6)(s+7/6)\bigr]^4
    \\
    & \bigl[(s+7/10)(s+9/10)(s+11/10)(s+13/10)\bigr]^2.
  \end{align*}

  \noindent (5) $(\tG = \GL(n) \times \SL(n) : V =
  \Mat_n(\C))$, $W = \Z_n$, $f = \det(x_{ij})$, $\Delta=
  \det(\partial_{x_{ij}})$, $b(s) = \prod_{i=0}^{n-1}(s+i
  +1)$.

  \noindent (6) $(\tG = \Sp(n) \times \SO(3) \times \C^* : V
  = \Mat_{2n,3}(\C))$, $W = \Z_4$, $\deg f =4$, $b(s) =
  (s+1)(s+3/2)(s+n)(s+n+1/2)$.

  \noindent Remark: The first five examples are regular
  irreducible PHV, but (6) gives is an example of an
  irreducible PHV which is not regular~\cite{Kim}. The
  description of the Verma modules on $U$ associated to
  examples~(1) to (4) are given in \S\ref{ssec22}.
\end{examples}

Let $\Theta$ be the Euler vector field on $V$; thus
$\Theta(p) = np$ for all $p \in S^n(V^*)$. In particular
$\Theta(f) = nf$, which implies that $\rad(\Theta) =
n\theta$. Set
\begin{equation}
  \label{theta}
  \bTheta = \frac{1}{n}\Theta,
\end{equation}
so that $\rad(\bTheta) = \theta$.

\begin{lem}
  \label{lem33}
  One has:
  $$
  U \subset R, \qquad U[z^{-1}] = R[z^{-1}] = \C[z^{\pm
    1},\dz].
  $$
\end{lem}

\begin{proof}
  Let $\Delta = f^*(\partial) \in S(V)$ be as in
  in~\eqref{eq35}. By definition and Theorem~\ref{thm31} we
  obtain: $\rad(\Delta)(z^{s+1}) = \psi(\Delta(f^{s+1})) =
  b(s)z^{s}$. Hence $\rad(\Delta) = \delta$, see
  Proposition~\ref{prop28}. From $\rad(f) = z$ and
  $\rad(\bTheta) = \theta$ it follows that $U =
  \C[z,\delta,\theta] \subset R \subset \C[z,\dz]$.  Observe
  that $U[z^{-1}] = \C[z^{\pm 1}, \delta, z^{-1}\theta =
  \dz,] = \C[z^{\pm 1},\dz] \subset R[z^{-1}] \subset
  \C[z^{\pm 1},\dz]$, thus $U[z^{-1}] = R[z^{-1}] =
  \C[z^{\pm 1},\dz]$.
\end{proof}
 
Recall that $\phi : \CV \to \C[V^*]$, $\phi(p) = p^*$, is
the bijective $\tK$-equivariant $\C$-antilinear morphism
defined in~\eqref{eq32}.

\begin{lem}
  \label{lem34}
  The map $\phi$ extends to a $\tK$-equivariant
  $\C$-antilinear anti-auto\-mor\-phism of $\DV$, given by
  $\phi(p) = \partial(p^*)$. One has:
  $$
  \phi\bigl(\DVG\bigr) = \DVG, \quad \phi^2= \id, \quad
  \phi(f) = \Delta = \partial(f^*), \quad \phi(\bTheta) =
  \bTheta.
  $$
\end{lem}

\begin{proof}
  In the coordinate system $\{z_i,\partial_{z_i}\}_{1 \le i
    \le N}$ we have: $\phi(z_i)= \partial_{z_i}$,
  $\phi(\partial_{z_i}) = z_i$, $\phi(a) = \overline{a}$ for
  $a \in \C$. Since $\DV =\C[z_i, \partial_{z_j} : 1 \le i,j
  \le N]$ with relations $[\partial_{z_j},z_i] =
  \delta_{ij}$, it is clear that $\phi$ extends to a
  $\C$-antilinear anti-automorphism of $\DV$ such that
  $\phi^2= \id$. By construction $\phi$ is
  $\tK$-equivariant, in particular $K$-equivariant if $K
  \subset \tK$ is a maximal compact subgroup of $G$. From
  $G=K_\C= K \exp(i\k)$ it follows that $\DV^K = \DV^G$,
  hence $\phi\bigl(\DVG\bigr) = \DVG$. The equality $\phi(f)
  = \Delta$ is obvious and $\phi(\bTheta) = \bTheta$ is
  consequence of $\Theta= \sum_i z_i \partial_{z_i}$.
\end{proof}

Recall that $\rad : \DVG \sto R $; we now want to check that
the anti-automorphism $\phi$ induces an anti-automorphism on
$R$ such that $\phi(z) = \delta$.  Denote by $J$ the kernel
of $\rad$, thus:
$$
J = \bigl\{D \in \DVG : D(f^m) = 0 \ \text{for all $m \in
  \N$}\bigr\}.
$$
Set
\begin{equation}
  \label{eq36}
  \D = \DVG \supset \tD = \mathcal{D}(V)^{\tG}.
\end{equation}
Since $\Theta \in \DVG$ we can decompose $\D$ under the
adjoint action of $\Theta$:
\begin{equation}
  \label{eq37}
  \D = \bigoplus_{p \in \Z} \, \D[p], \quad \D[p] = \{D \in \D :
  [\Theta,D] = pD\}.
\end{equation}

\begin{prop}
  \label{prop35}
  One has $\phi(J)=J$.
\end{prop}

\begin{proof}
  As $\phi^2 = \id$ we need to show that $\phi(J) \subset
  J$. Since $J$ is an ideal of $\DVG$ it decomposes under
  the adjoint action of $\Theta$: $J = \boplus_{p \in \Z}
  J[p]$, $J[p] = J \cap \D[p]$. Thus we only need to check
  that $\phi(J[p]) \subset J$.  In the previous coordinate
  system $\{z_i,\partial_{z_i}\}_{1 \le i \le N}$ we have:
  $$\D[p] = \sum_{|\mathtt{i}| - |\ttj| = p} \C z^\tti
  \partial^\ttj.
  $$
  We can write $D \in \D[p]$ in a unique way under the form
  $$
  D= D_0 + \cdots+D_t, \quad D_k = \sum_{|\ttj| = k
    -p}\biggl(\sum_{|\tti| = k}a_{\tti,\ttj} z^\tti\biggl)
  \partial^\ttj,
  $$
  (thus $D_k = 0$ when $k < p$). If $D_t \ne 0$, we set $t=
  \deg_z D$.
  
  Let $D =\sum_{|\mathtt{i}| - |\ttj| = p} a_{\tti,\ttj}
  z^\tti \partial^\ttj$ be in $J[p] \smallsetminus\{0\}$.
  As $G$ acts linearly on $V$ we have $G.D_k \subset
  S^k(V^*)S^{k -p}(V)$.  Since $D$ is $G$-invariant it
  follows that each $D_k$ is $G$-invariant.  We have:
  $$
  \phi(D) = \sum_k \phi(D_k), \quad \phi(D_k) =
  \sum_{\{|\tti|=k, |\ttj| = k -p\}}
  \overline{a_{\tti,\ttj}} z^{\ttj}\partial^\tti.
  $$
  From the previous expression we get that $\phi(D_k)(1) =
  0$ when $k >0$, hence
  $$
  \phi(D)(1) = \phi(D_0)(1) = \phi(D_0) = \sum_{|\ttj| = -p}
  \overline{a_{0,\ttj}} z^{\ttj}.
  $$
  Assume that $D_0 = \sum_{|\ttj| = -p} a_{0,\ttj}
  \partial^\ttj \in S^p(V)^G$ is non zero. From
  Proposition~\ref{prop32}(ii) we can deduce that $p = -\ell
  n$, $\ell \ge 0$, $D_0 = \alpha\Delta^{\ell}$, $\alpha \ne
  0$.  By hypothesis $D(f^\ell) = \sum_k D_k(f^\ell) = 0$.
  Note that, since $|\ttj| = k + n\ell > n\ell$ implies
  $\partial^\ttj(f^\ell) =0$, we have $D_k(f^\ell) =
  \sum_{\{|\tti|=k, |\ttj| = k +n\ell\}} a_{\tti,\ttj}
  z^{\tti}\partial^\ttj(f^\ell) = 0$ if $k > 0$. Thus
  $D_0(f^\ell)=D(f^\ell) = 0$. If $\ell =0$ we get $D_0
  =\alpha = D_0(1) =D_0(f^\ell) =0$, contradiction.
  Therefore $\ell \ge 1$ and
  $$
  0 = D(f^\ell) = D_0(f^\ell) = \alpha\Delta^\ell(f^\ell) =
  \alpha b(\ell-1) \cdots b(0).
  $$
  It is easily seen that $\Delta^\ell(f^\ell) \ne 0$, see
  \cite[Proof of Proposition~2.22]{Kim} (this is equivalent
  to $b(j) \ne 0$ for all $j \in \N$), hence a
  contradiction. Thus: $\phi(D)(1) = \phi(D_0) = D_0=0$.
  
  We show that $\phi(D) \in J$ by induction on $t = \deg_z
  D$.  (In the case $t=0$ one has $D=D_0=0$.) Since $\Delta
  \in S(V)^G$ and $J$ is an ideal, one has $[D,\Delta] \in
  J$. Observe that
  $$
  [D,\Delta] = \sum_k [D_k,\Delta], \quad [D_k,\Delta] =
  \sum_{|\ttj| = k -p}\bigl(\sum_{|\tti| = k}a_{\tti,\ttj}
  [z^\tti,\Delta] \bigr) \partial^\ttj.
  $$
  But $\Delta \in S(V)^G$ implies that
  $\deg_z[z^\tti,\Delta] < k = |\tti|$, hence
  $\deg_z[D,\Delta] < t=\deg_z D$. Then, by induction,
  $\phi([D,\Delta]) = [\phi(\Delta),\phi(D)] = [f,\phi(D)]
  \in J$. If $m \ge 0$ we then have: $0 = [f,\phi(D)](f^m) =
  f \phi(D)(f^m) - \phi(D)(f^{m+1})$. This implies, by
  induction on $m$, $\phi(D)(f^{m+1}) = f\phi(D)(f^m) =
  f^{m+1}\phi(D)(1)$. It follows from the previous paragraph
  that $\phi(D)(f^{m+1}) = \phi(D)(1) =0$, i.e. $\phi(D) \in
  J$.
\end{proof}

\begin{cor}
  \label{cor36} {\rm (1)}
  There exists a $\C$-antilinear anti-auto\-mor\-phism $\phi
  : R \to R$ such that:
  $$
  \phi^2= \id, \quad \phi(z) = \delta, \quad \phi(\theta) =
  \theta, \quad \phi(U) = U.
  $$
  {\rm (2)} One has $U[\delta^{-1}] = R[\delta^{-1}]$.
\end{cor}

\begin{proof}
  (1) Let $\phi : \DVG \to \DVG$ be as in Lemma~\ref{lem34}.
  By Proposition~\ref{prop35} we can define $\phi : R \to R$
  by setting
  $$
  \phi(\rad(D)) = \rad(\phi(D)).
  $$
  Indeed: if $\rad(D) = \rad(D')$ we get $D- D' \in J =
  \Ker(\rad)$, hence $\phi(D) - \phi(D') \in J$ and
  $\rad(\phi(D)) = \rad(\phi(D'))$. The equality $\phi^2=
  \id$ is clear; by definition and Lemma~\ref{lem34}:
  $$
  \phi(z) = \rad(\phi(f)) = \rad(\Delta) = \delta, \quad
  \phi(\theta) = \rad(\phi(\bTheta)) = \rad(\bTheta) =
  \theta.
  $$
  From $U=\C[z,\delta,\theta]$ we then deduce $\phi(U)=U$.

  \noindent (2) Observe that $\ad(\phi(u))^m(r) = (-1)^m
  \phi(\ad(u)^m(r))$ for all $u,r \in R$. Since $\ad(z)$ is
  a locally nilpotent operator in $R$, it follows that
  $\ad(\phi(z))=\ad(\delta)$ has the same property. We can
  therefore construct the $\C$-algebras $U[\delta^{-1}]
  \subset R[\delta^{-1}]$.

  Let $Q= \Frac(U)$ be the fraction field of the Noetherian
  domain $U$.  By Lemma~\ref{lem33} we know that
  $Q=\C(z,\dz) = \Frac(R)$. It is easy to see that $\phi$
  extends to $Q$ by $\phi(s^{-1}a) = \phi(a)\phi(s)^{-1}$
  for all $a\in r$, $0 \ne s \in R$. This gives a
  $\C$-antilinear anti-auto\-mor\-phism of $Q$. Then
  $\phi(R[z^{-1}])) = \phi(U[z^{-1}]) = \phi(U)[\delta^{-1}]
  = U[\delta^{-1}]$, which yields
  $U[\delta^{-1}]=\phi(R[z^{-1}]) = \phi(R)[\delta^{-1}]=
  R[\delta^{-1}]$, as desired.
\end{proof}

Let $M$ be a module over a $\C$-algebra $A$, then the
Gelfand-Kirillov of $M$ is denoted by $\gk_A M$ or simply
$\gk M$, see~\cite{MR}.

\begin{lem}
  \label{lem37}
  Let $r \in R$. Then:
  $$
  \gk_U (U + Ur)/U \le \gk_U U - 2 =0.
  $$
\end{lem}

\begin{proof}
  From Corollary~\ref{cor36} we deduce that there exists
  $\nu \in \N$ such that $z^\nu r \in U$ and $\delta^\nu r
  \in U$. Therefore the $U$-module $(U + Ur)/U$ is a factor
  of $U/(U z^\nu + U \delta^\nu)$.  There exists on $U \cong
  \widetilde{U}/(\Omega)$ (cf.~Proposition~\ref{prop29}) a
  finite dimensional filtration such that $\gr(U)$ is
  isomorphic to the commutative algebra $\C[X,Y,S]/(XY -
  S^n)$, see \S\ref{ssec21} or \cite{Sm, MVdB}, where
  $\gr(z) = X$, $\gr(\delta) =Y$.
  % For this filtration $\gr(z) = X$, $\gr(\delta) = Y$ and
  % $z, \delta, \theta$ have respective degrees $1, n-1, 1$.
  It follows that the associated graded module of $U/(U
  z^\nu + U \delta^\nu)$ is a factor of $\gr(U)/(\gr(U)
  X^\nu + \gr(U)Y^\nu)$, which is finite dimensional. Hence
  the result.
\end{proof}

We now can prove the main result of this section.

\begin{thm}
  \label{thm38}
  One has $U=R$.
\end{thm}

\begin{proof}
  Endow $U$ with a filtration such that $\gr(U) \cong
  \C[X,Y,S]/(XY - S^n)$ as in the proof of the previous
  lemma. Observe that $\C[X,Y,S]/(XY - S^n)$ is commutative
  Gorenstein normal domain. By \cite[Theorem~3.9]{Bj2} $U$
  is Auslander-Gorenstein and by \cite{VdBVO} $U$ is a
  maximal order. Recall that $Q = \Frac(U)$ and consider the
  following family of finitely generated $U$-modules $M$:
  $$
  \calF = \bigl\{ U \subset M \subset Q : \gk_U M/U \le
  \gk_U U - 2\bigr\}.
  $$
  From \cite[Theorem~1.14]{Bj2} we know that $\calF$
  contains a unique maximal element $\tilde{M}$. By
  Lemma~\ref{lem37} we have $U + Ur \subset \tilde{M}$ for
  all $r \in R$; hence $R \subset \tilde{M}$. It follows
  that $R$ is finitely generated over $U$ with
  $Q=\Frac(U)=\Frac(R)$.  Thus $U=R$, since $U$ is a maximal
  order.
\end{proof}

\begin{rem}
\label{rem39}
Let $(G : V)$ be a representation of the connected reductive
group $G$ such that $\dim V\qmod G =1$. If $\CVG= \C[f]$ one
can define $\Delta \in S(V)^G$ and the polynomial $b(s) =
c\prod_{i=0}^{n-1} (s + \lambda_i +1)$ as in
Theorem~\ref{thm31}. Then, the proof of Theorem~\ref{thm38}
can be repeated to show that $R= \Im(\rad)= U(k)$ (where
$k_i = \lambda_i +\frac{i}{n}$, $0 \le i \le n-1$).
\end{rem}

%%%%%%%%%%%%%%%%%%%%%%
\section{Multiplicity free representations}
\label{sec4}
%%%%%%%%%%%%%%%%%%%%%%%%%

\subsection{Generalities}
\label{ssec41}
Let $\tGV$ be a connected reductive group. Write $\tG = GC$,
$C \cong (\C^*)^c$, as in \S\ref{ssec31}. We adopt the
following notation:
\begin{itemize}
\item the Lie algebra of an algebraic group is denoted by
  the corresponding gothic character;
\item $TU$ is a Borel subgroup of $G$, $T$ being a maximal
  torus of $G$, hence $\tT U$ is a Borel subgroup of $\tG$,
  $\tT= TC$;
  
\item $\Rt$ is the root system of $(\g,\ft)$,
  $\Bt=\{\alpha_1,\dots,\alpha_\ell\}$ is a basis of $\Rt$
  and $\Rt^+$ is the set of associated positive roots;
\item $\Lambda$ is the weight lattice of $(\g,\ft)$, thus
  $\Lambda = \Z\vpi_1 \boplus \cdots \boplus \Z\vpi_\ell$
  where $\dual{\vpi_i}{\alpha_j} = \delta_{ij}$; $\Lambda^+
  = \N\vpi_1 \boplus \cdots \boplus \N\vpi_\ell$ denotes the
  dominant weights;
\item $\tLambda= \Lambda \boplus \Xt(C)$, with $\Xt(C) \cong
  \Z^c$; $\tLambda^+ = \Lambda^+ \boplus \Xt(C)$;
\item if $\tlambda \in \tLambda^+$, resp.~$\lambda \in
  \Lambda^+$, we denote by $E(\tlambda)$,
  resp.~$E(\lambda)$, an irreducible $\tfg$-module,
  resp.~$\g$-module, with highest weight $\tlambda$,
  resp.~$\lambda$; the dual of $E(\tlambda)$ is isomorphic
  to $E(\tlambda^*)$, $\tlambda^*= - w_0(\tlambda)$ where
  $w_0$ is the longest element of the Weyl group of $\Rt$
  (similarly for $\Elambda^*$).
  %% \item if $M$ is a finite dimensional $\tG$-module,
  %%   $m(\tlambda) = [M : \tlambda]$ is the multiplicity of
  %%   $\Etlambda$ in $M$.
\end{itemize}

We fix a finite dimensional representation $\tGV$ of the
reductive group $\tG$.  Then the rational $\tG$-module $\CV=
S(V^*)$ decomposes as
$$
\CV \cong \bigoplus_{\tlambda \in \tLambda^+}
E(\tlambda)^{m(\tlambda)}
$$
% where $\Vtlambda$ is a $\tG$-module isomorphic to
% $\Etlambda$ and
where $m(\tlambda) \in \N \cup \{\infty\}$.

\begin{defn}
  \label{def41}
  The representation $\tGV$ is called multiplicity free (MF
  for short) if $m(\tlambda) \le 1$ for all $\tlambda$.  In
  this case
  $$
  \CV = \bigoplus_{\tlambda \in \tLambda^+}
  V(\tlambda)^{m(\tlambda)}, \quad m(\tlambda) = 0,1
  $$
  where $\Vtlambda \subset S^{d(\tlambda)}(V^*)$ is
  isomorphic to $\Etlambda$; if $m(\tlambda) =1$,
  $d(\tlambda)$ is called the degree of $\tlambda$ in $\CV$.
\end{defn}

\begin{remark}
  The MF representations are classified \cite{Kac, BR, Lea}.
  We give in Appendix~\ref{A1} the list of $\tGV$ with $V$
  irreducible (see \cite{Kac}). For instance, the examples
  (1), (2), (3), (5) given in \S\ref{ssec32} are MF.
\end{remark}

{\em From now on, let $\tGV$ be a MF representation.}  The
following results can be found, for example, in \cite{BR,
  HU, Kac, Kn1}.%, Kn2}.

-- Set $\tGamma = \{\tlambda : m(\tlambda) = 1\}$, then
$\tGamma = \boplus_{i=0}^r \N \tlambda_i$ where the
$\tlambda_i$ are linearly independent over $\Q$.

-- The algebra of $U$-invariants $\CV^U = \C[h_0,\dots,h_r]$
is a polynomial ring. If $\tgamma = \sum_i m_j \tlambda_j
\in \tGamma$, one has $\Vtgamma = U(\tfg).h^{\tgamma}$ where
$h^{\tgamma} = h_0^{m_0}\cdots h_r^{m_r}$ is a highest
weight vector of $\Vtgamma$.  In particular: $h_j =
h^{\tlambda_j}$, $d(\tgamma)= \sum_j m_j d(\tlambda_j)$.

-- The representation $\tGV$ is a prehomogeneous vector
space. Let $f_0,\dots,f_m$ be the basic relative invariants
of this PHV and let $\chi_j \in \Xt(\tG) = \Xt(C)$, $0 \le j
\le m$, be their weights. After identification of $\Xt(C)$
with a subgroup of $\tLambda$ as above, one can number the
$\tlambda_j$ so that
$$
\tlambda_0 \equiv \chi_0, \dots, \tlambda_m \equiv \chi_m,
\quad h_0= f_0, \dots,h_m = f_m,
$$
thus $V(\tlambda_j)= V(\chi_j)$ is the one dimensional
$\tG$-module $\C f_j$.

\medskip

Let $p : \tLambda^+ = \Lambda^+ \boplus \Xt(C) \sto
\Lambda^+$
% and $q : \tLambda^+ \sto \Xt(C)$
be the natural projection. Set
\begin{equation}
  \label{eq41}
  \tGamma= \Gamma_0 \bigoplus \Gamma, \quad \Gamma_0 =
  \bigoplus_{j=0}^m \N\tlambda_j= \bigoplus_{j=0}^m 
  \N\chi_j, \quad \Gamma = \bigoplus_{j=m+1}^r \N\tlambda_j.
\end{equation}
Using the results above, the next lemma is easy to prove.

\begin{lem}
  \label{lem42}
  One has:

  \noindent {\rm (a)} $\Gamma_0= \Xt(C) \cap \tGamma =
  \{\tgamma \in \tGamma : \tgamma(\ft) = 0\}$; $p$ induces a
  bijection $\Gamma \isomto p(\Gamma)$;

  \noindent {\rm (b)} let $\tgamma \in \tGamma$, then the
  $G$-module $V(\tgamma)$ is isomorphic to $E(p(\tgamma))$;

  \noindent {\rm (c)} let $\gamma, \gamma' \in \Gamma$, then
  the following are equivalent:
  \begin{enumerate}
  \item[{\rm (i)}] $\gamma = \gamma'$
  \item[{\rm (ii)}] $p(\gamma)= p(\gamma')$
  \item[{\rm (iii)}] $V(\gamma) \cong V(\gamma')$ as
    $G$-modules;
  \end{enumerate}
  \noindent {\rm (d)}
  the algebra $\CV^G$ of $G$-invariants is polynomial ring,
  more precisely:
  $$
  \CV^G= \C[f_0,\dots,f_m] = \bigoplus_{\gamma \in \Gamma_0}
  \C h^\gamma.
  $$
\end{lem}

Set:
\begin{equation}
  \label{eq42}
  H(V^*)  = \bigoplus_{\gamma \in \Gamma} V(\gamma).
\end{equation}

\begin{lem}
  \label{lem43}
  The multiplication map:
  $$
  \msf : H(V^*) \otimes_\C S(V^*)^G \lto S(V^*)= \CV
  $$
  is an isomorphism of $G$-modules.
\end{lem}

\begin{proof}
  Let $\tgamma = \gamma + \gamma_0$, $\gamma \in \Gamma,
  \gamma_0 \in \Gamma_0$. Observe that $C$ acts by scalars
  on the simple $\tG$-module $V(\tgamma)$; thus, since
  $\tfg= \g \boplus \fc$, we have: $V(\tgamma) =
  U(\tfg).h^{\tgamma} = U(\g)U(\fc).h^\tgamma =
  U(\g).h^\tgamma = U(\g).h^\gamma h^{\gamma_0} =
  (U(\tfg).h^\gamma) h^{\gamma_0} = V(\gamma)h^{\gamma_0}$.
  Therefore $V(\tgamma) =\msf(V(\gamma) \otimes \C
  h^{\gamma_0})$ with $V(\gamma) \subset H(V^*)$,
  $h^{\gamma_0} \in S(V^*)^G$. Suppose that $\msf(\sum_i v_i
  \otimes h^{\mu_i}) =0$ with $\mu_i \in \Gamma_0$, $v_i \in
  V(\lambda_i)$, the $\lambda_i \in \Gamma$ being pairwise
  distinct. Observe that $\lambda_i + \mu_i = \lambda_j +
  \mu_j$ forces $\lambda_i - \lambda_j = \mu_j - \mu_i \in
  \Gamma \cap \Gamma_0 = (0)$. Therefore $v_ih^{\mu_i} \in
  V(\lambda_i + \mu_i)$ and $\sum_i v_ih^{\mu_i } = 0$ yield
  $v_ih^{\mu_i} =0$, hence $v_i =0$ for all $i$.
\end{proof}

Recall that we identify $S(V)$ with the algebra of
differential operators with constant coefficients. Consider
the non-degenerate pairing
$$
S(V) \otimes S(V^*) \lto \C, \quad u \otimes \vphi \mapsto
\ascal{u}{\vphi} = u(\vphi)(0),
$$
which extends the duality pairing $V \otimes V^* \to \C$.
It is easily shown that:
\begin{itemize}
\item $\ascal{u}{S^j(V^*)} = 0$ if $u \in S^i(V)$ and $i \ne
  j$;
\item $\ascal{\phantom{s}}{\phantom{s}}$ is
  $\tG$-equivariant.
\end{itemize}
Therefore $u \mapsto \ascal{u}{\phantom{s}}$ gives a
$\tG$-isomorphism from $S^i(V)$ onto $S^i(V^*)^*$. In
particular, the representation $\tGVd$ is MF and we can
write:
$$
S^i(V) = \bigoplus_{\{\tgamma \in \tGamma, d(\tgamma)= i\}}
Y(\tgamma), \quad Y(\tgamma) \cong V(\tgamma)^* \cong
E(\tgamma^*).
$$
Hence, $Y(\tgamma)=U(\tfg).\Delta^\tgamma $ where
$\Delta^\tgamma$ is a lowest weight vector (of weight
$-\tgamma$). When $\tgamma = \tlambda_j$ we set
$\Delta^{\tlambda_j} =\Delta_j$. Note that $\Delta^\tgamma =
\prod_{i=0}^r \Delta_i^{m_i}$ if $\tgamma = \sum_i m_i
\tlambda_i$.

If $0 \le i \le m$ we have $Y(\tlambda_i) = \C \Delta_i$
where $\Delta_i$ has weight $-\lambda_i \equiv \chi_i^{-1}$.
Clearly, we may take $\Delta_i = \partial(f_i^*)$ where
$f_i^*$ is the relative invariant constructed as in
\S\ref{ssec31}. We then have
\begin{equation}
  \label{eq43}
  S(V)^G= \C[\Delta_0,\dots,\Delta_m] = \bigoplus_{\gamma \in
    \Gamma_0} \C \Delta^\gamma
\end{equation}
(which is a polynomial ring).

If $\mu = \sum_i m_i \tlambda_i$ and $\nu = \sum_i n_i
\tlambda_i$ are elements of $\tGamma$, we say that $\mu \le
\nu$ if $m_i \le n_i$ for all $i$. Let $k : \tGamma \sto
\Gamma_0$ be the projection associated to the decomposition
defined in~\eqref{eq41}; thus each $\tlambda \in \tLambda$
writes uniquely $\gamma + k(\tlambda)$, $\gamma \in \Gamma$,
$k(\tlambda) \in \Gamma_0$.

\begin{lem}
  \label{lem44}
  Let $\tlambda\in \Gamma_0$ and $\tgamma \in \tGamma$.
  Then:
  \begin{enumerate}
  \item[{\rm (a)}] $\Delta^\tgamma(h^\tgamma) \ne 0$;
  \item[{\rm (b)}] $\Delta^\tlambda(h^\tgamma) \ne 0 \iff
    \tlambda \le k(\tgamma)$, and in this case
    $\Delta^\tlambda$ gives an isomorphism of $G$-modules,
    $\vphi \mapsto \Delta^\tlambda(\vphi)$, from $\Vtgamma$
    onto $V(\tgamma - \tlambda)$.
  \end{enumerate}
\end{lem}

\begin{proof}
  Set $\tlambda= \sum_{i=0}^m p_i \tlambda_i$, $\tgamma =
  \sum_{i=0}^r q_i \tlambda_i$.
  
  (a) Recall that we have an isomorphism of $\tG$-modules,
  $\beta : Y(\tgamma) \isomto V(\tgamma)^*$, $\beta(u) =
  \ascal{u}{\phantom{s}}$. Thus $\beta(\Delta^\tgamma)$ is a
  lowest vector in $V(\tgamma)^*$, which implies
  $\beta(\Delta^\tgamma)(h^\tgamma) =
  \Delta^\tgamma(h^\tgamma)(0) \ne 0$. But $\Delta^\tgamma
  \in S^{d(\tgamma)}(V)$ where $d(\tgamma)$ is the degree of
  $\tgamma$, therefore $\Delta^\tgamma(h^\tgamma) \in \C$.
  Thus $\Delta^\tgamma(h^\tgamma) =
  \Delta^\tgamma(h^\tgamma)(0) \ne 0$.
  
  (b) Since $\Delta^\tlambda \in S(V)^G$ we have
  $\Delta^\tlambda(V(\tgamma)) =
  \Delta^\tlambda(U(\g)h^\tgamma) =
  U(\g)\Delta^\tlambda(h^\tgamma)$. By Lemma~\ref{lem42} we
  know that $V(\tgamma)$ is a simple $G$-module, it follows
  that the map $\Delta^\tlambda : V(\tgamma) \to
  \Delta^\tlambda(V(\tgamma))$ is either $0$ or an
  isomorphism of $G$-modules.
  
  Notice that $\Delta^\tlambda(h^\tgamma) \in \CV^U$ has
  weight $\tgamma - \tlambda = k(\tgamma) -\tlambda +
  \sum_{i=m+1}^r q_i \tlambda_i$ where $k(\tgamma) -\tlambda
  = \sum_{i=0}^m (q_i - p_i) \tlambda_i$. Therefore if $q_i
  < p_i$ for some $i =0,\dots,m$ we must have
  $\Delta^\tlambda(h^\tgamma)=0$,
  i.e.~$\Delta^\tlambda(h^\tgamma) \ne 0$ implies $\tlambda
  \le k(\tgamma)$. Conversely, suppose that $\tlambda \le
  k(\tgamma)$; then, by (a),
  $$
  0 \ne \Delta^\tgamma(h^\tgamma) = \Delta^{\tgamma -
    \tlambda}\Delta^\tlambda(h^\tgamma)
  $$
  which forces $\Delta^\tlambda(h^\tgamma) \ne 0$.
  
  Now assume $\tlambda \le k(\tgamma)$. Then $0 \ne
  \Delta^\tlambda(h^\tgamma) \in \CV^U$ implies that
  $\Delta^\tlambda(h^\tgamma)$ is a highest weight vector in
  $V(\tgamma - \tlambda)$, hence $\Delta^\tlambda :
  V(\tgamma) \longisomto \Delta^\tlambda(V(\tgamma)) =
  V(\tgamma - \tlambda)$.
\end{proof}

Recall the definition of $H(V^*)$ given in~\eqref{eq42} and
set $S_+(V) = \boplus_{i >0} S^i(V)$.  The next proposition
identifies $H(V^*)$ with harmonic elements.

\begin{prop}
  \label{prop45}
  We have:
  \begin{align*}
    H(V^*) & = \{\vphi \in \CV : \Delta_0(\vphi)= \cdots =
    \Delta_m(\vphi)=0\}
    \\
    & = \bigl\{\vphi \in \CV : D(\vphi) = 0 \ \text{for all
      $D \in S_+(V)^G$} \bigr\}.
  \end{align*}
\end{prop}

\begin{proof}
  From~\eqref{eq43} we know that $S_+(V)^G = \bigoplus_{0
    \ne \tlambda \in \Gamma_0} \C \Delta^\tlambda$.  Let
  $\vphi \in V(\tgamma)$ for some $\tgamma \in \Gamma$ and
  let $0 \ne \tlambda \in \Gamma_0$. We have $k(\tgamma) =
  0$, thus $\Delta^\tlambda(\vphi) = 0$ by
  Lemma~\ref{lem44}(b).  This shows that $H(V^*) \subset
  \bigl\{\vphi \in \CV : D(\vphi) = 0 \ \text{for all $D \in
    S_+(V)^G$} \bigr\}$.
  
  Conversely assume that $\vphi = \sum_{\tgamma \in \tGamma}
  \vphi_{\tgamma}$, $\vphi_{\tgamma} \in \Vtgamma$,
  satisfies $\Delta_i(\vphi) = 0$ for all $i=0,\dots,m$.
  Fix $i \in \{0,\dots,m\}$. By Lemma~\ref{lem44}(b) we get
  that $\Delta_i(V(\tgamma)) =0$ if $\tlambda_i \not\le
  k(\tgamma)$ and $\Delta_i : \Vtgamma \isomto V(\tgamma -
  \tlambda_i)$ if $\tlambda_i \le k(\tgamma)$. Therefore
  $\Delta_i(\vphi) = \sum_{\tgamma \in \tGamma}
  \Delta_i(\vphi_{\tgamma})$ belongs to
  $\bigoplus_{\{\tgamma \in \tGamma, \tlambda_i \le
    k(\tgamma)\}} V(\tgamma - \tlambda_i)$. Since
  $\Delta_i(\vphi) =0$ we can deduce that
  $\Delta_i(\vphi_{\tgamma}) =0$ for all $\tgamma$ such that
  $\tlambda_i \le k(\tgamma)$. By the previous remark this
  implies $\vphi_{\tgamma} =0$ when $\tlambda_i \le
  k(\tgamma)$, thus $\vphi = \sum_{\{\tgamma \in \tGamma,
    \tlambda_i \not\le k(\tgamma)\}}\vphi_{\tgamma}$.
  Observe that $\tlambda_i \not\le k(\tgamma)$ means that
  the weight $\tlambda_i$ does not appear in $\tgamma$.
  Since this holds for all $i=0,\dots,m$ we deduce that
  $\vphi = \sum_{\tgamma \in \Gamma}\vphi_{\tgamma}$. Hence
  the result.
\end{proof}

\begin{rem}
  \label{rem46}
  If we set $H(V)= \bigoplus_{\gamma \in \Gamma} Y(\gamma)$
  we obtain that $S(V) \cong H(V) \otimes S(V)^G$ as
  $G$-modules, with an analogous characterization of $H(V)$.
\end{rem}

We now recall some facts about invariant differential
operators on MF representations, cf.~\cite{BR, HU, Kn1}.
Recall that the $\CV$-module $\DV$ identifies with $S(V^*)
\otimes S(V)$ through the multiplication map
$$
\msf : S(V^*) \otimes S(V) \isomto \DV, \quad \vphi \otimes
f \mapsto \vphi f(\partial).
$$
The isomorphism $\msf$ is also $\tG$-equivariant, hence
$\DV^{\tG} \cong \bigoplus_{\tgamma \in \tGamma} [V(\tgamma)
\otimes Y(\tgamma)]^{\tG}$. But, since $Y(\tgamma) \cong
\Vtgamma^*$, $[V(\tgamma) \otimes Y(\tgamma)]^{\tG} = \C
E_\tgamma$ is one dimensional.  Let
$$
E_\tgamma(x,\dx) = \frac{1}{\dim V(\tgamma)}\msf(E_\tgamma)
\in \DV^{\tG}
$$
be the operator corresponding to $E_\tgamma$. The
$E_\tgamma(x,\partial_x)$ are called the {\em normalized
  Capelli operators}. Set

\begin{equation}
  \label{eq44}
  E_j = E_{\tlambda_j}(x,\partial_x), \quad 0 \le j \le r. 
\end{equation}

It is known \cite[Proposition~7.1]{HU} that $\tGV$
multiplicity free is equivalent to $\DV^{\tG}$ commutative.
The operators $E_j$ give a set of generators for this
algebra, cf.~\cite[Theorem~9.1]{HU} or
\cite[Corollary~7.4.4]{BR}:

\begin{thm}[Howe-Umeda]
  \label{thm46}
  $\tD = \DVtG = \C[E_0,\dots,E_r] = \bigoplus_{\tgamma \in
    \tGamma} \C E_\tgamma(x,\dx)$ is a commutative
  polynomial ring.
\end{thm}

Notice for further use the following property of the Capelli
operators, see \cite[Corollary~4.4]{Kn1} or
\cite[Proposition~8.3.2]{BR}:

\begin{prop}
  \label{prop47}
  Set $\fa^*= \C \otimes_\Z \Z\tGamma = \boplus_{i=0}^r\C
  \tlambda_i$, $\fa= \boplus_{i=0}^r\C a_i$ where
  $\{a_i\}_i$ is the dual basis of $\{\tlambda_i\}_i$. For
  each $\tgamma \in \tGamma$ there exists a polynomial
  function $b_\tgamma = b_\tgamma(a_0,\dots,a_r) \in
  \C[\fa^*] = S(\fa)=\C[a_0,\dots,a_r]$ such that
  $$
  E_\tgamma(x,\dx)(h^\tlambda) =
  b_\tgamma(\tlambda)h^\tlambda \ \; \text{for all $\tlambda
    \in \fa^*$}.
  $$
\end{prop}

\begin{rems}
  \label{rem48}
  (1) Suppose that $\tGV$ is irreducible. Then we can assume
  that $V = V(\tgamma_r)$. If $\dim V= N$ we have
  $E_r=\Ej{r} =\bTheta= \frac{1}{N} \Theta$ where $\Theta$
  is the Euler vector field.
  
  (2) If $j \in \{0,\dots,m\}$ we may take $E_j =
  f_j\Delta_j$. Recall that $f_j = h^{\tlambda_j}$ and
  $\Delta_j = \partial(f_j^*)$. By Theorem~\ref{thm31} there
  exists $b_j(s) \in \C[s]$ such that $\Delta_j(f_j^m) =
  b_j^*(m) f_j^{m-1}$; thus $E_j(f_j^m) = b_j^*(m) f_j^m$.
  This shows that $b_{\tlambda_j}(s,0,\dots,0) = b_j^*(s)$.
  
  (3) Let $D \in \tD$; then $D(V(\tlambda)) =
  U(\tfg).D(h^{\tlambda})$ is either $(0)$ or equal to
  $\Vtlambda$. Indeed, the $\tG$-invariance of $D$ implies
  that $g.D(h^{\tlambda}) = D(g.h^{\tlambda})=
  \tlambda(g)D(h^{\tlambda})$ for all $g \in \tT U$, where
  we have considered here $\tlambda$ as a character of the
  Borel subgroup $\tT U$ of $\tG$; thus $D(h^{\tlambda}) \in
  \C h^{\tlambda}$ is either $0$ or a highest weight vector
  of $V(\tlambda)$.
\end{rems}

\subsection{MF representations with a one dimensional
  quotient}
\label{ssec42}

In this subsection we will work under the following
hypothesis:

\begin{mainhyp}
  \label{hypB}
  $\tGV$ is a multiplicity free representation which
  satisfies Hypothesis~\ref{hypA}.
\end{mainhyp}

In the notation of \S\ref{ssec41}, this condition means that
$m=0$, i.e.~$\Gamma_0 = \N \tlambda_0$. Set $f= f_0$, $n=
d(\tlambda_0)$, $\Delta = \Delta_0 = \partial(f^*)$, then we
have $f \in S^n(V^*)$, $V(\tlambda_0) = \C f$,
$Y(\tlambda_0) = \C \Delta$ and:
$$
\CVG = \C[f], \quad \SVG = \C[\Delta]
$$
(see Lemma~\ref{lem42} and~\eqref{eq43}). By
Remark~\ref{rem48}(2) we have $E_0 = f \Delta$,
$b_{\tlambda_j}(s,0,\dots,0) = b^*(s) = b(s-1) $ where
$b(s)$ is the $b$-function of $f$. Recall from~\eqref{eq36}
the following notation:
$$
\D = \DVG \supset \calD(V)^\tG = \tD.
$$
Recall also that $(G : V)$ is polar and that we have studied
in \S\ref{ssec32} the image of radial component map $\rad :
\DVG \to \DhmodW = \C[z,\dz]$. We now want to describe $J=
\Ker(\rad)$.

\begin{lem}
  \label{lem49}
  Let $P \in \tD$. Then there exists a polynomial $b_P(s)
  \in \C[s]$
  % of degree $d_P$
  such that
  $$
  P(f^m) = b_P(m) f^m, \quad \rad(P) = b_P(\theta), \quad P
  - b_P(\bTheta) \in J=\Ker(\rad).
  $$
\end{lem}

\begin{proof}
  Write $P= \sum_{\gamma \in \tGamma} p_\gamma
  \Ej{\tgamma}$, cf.~Theorem~\ref{thm46}, and define a
  polynomial function by $b_P(s) = \sum_{\tgamma \in
    \tGamma} p_\tgamma b_\tgamma(s,0,\dots,0)$, where
  $b_\tgamma \in S(\fa)$ is as in Proposition~\ref{prop47}.
  Since $f^m = h^{m\tlambda_0}$ we obtain that $P(f^m) =
  b_P(m) f^m$.  It follows that $\rad(P)(z^m) = b_P(m) z^m$
  for all $m \in \N$ and Lemma~\ref{lem21} yields $\rad(P)=
  b_P(\theta)$. Since $\rad(\bTheta) = \theta$ we have
  $\rad(P - b_P(\bTheta)) =0$.
\end{proof}

Notice that $\bTheta \in \tD$; for $j \in \{0,\dots,r\}$ we
set
\begin{equation}
  \label{eq45}
  \Omega_j = E_j - b_{E_j}(\bTheta) \in J \cap \tD. 
\end{equation}
Thus we have:
$$
\tD = \C[E_0,\dots,E_r] =
\C[\Omega_0,\Omega_1,\dots,\Omega_r,\bTheta].
$$
Recall that for $j=0$ one has $E_0 = f\Delta$, hence
$b_{E_0}(s) = b^*(s)$ where $b(s)$ is the $b$-function of
$f$. Thus $\Omega_0 = f\Delta - b^*(\bTheta)$; observe that
we have already shown in \S\ref{ssec32} that $\rad(f\Delta -
b^*(\bTheta)) = z\delta - b^*(\theta) = 0$.

When $V$ is irreducible we adopt the notation of
Remark~\ref{rem48}(1) and we obtain $E_r= \bTheta$,
$b_{E_r}(s) = s$, thus $\Omega_r =0$.  Therefore in this
case one has
\begin{equation}
  \label{eq46}
  \tD = \C[\bTheta,\Omega_0,\dots,\Omega_{r-1}].
\end{equation}

The next result gives a description of $\Ker(\rad)$ and
another proof of Theorem~\ref{thm38} in the case of MF
representations. When $\tGV = (\GL(n) : S^2\C^n)$, part~(i)
of Theorem~\ref{thm410} is proved in
\cite[Proposition~2.1]{Mu4}.

\begin{thm}
  \label{thm410}
  The following properties hold.

  \noindent {\rm (i)} $\D= \tD[f,\Delta] =
  \C[E_1,\dots,E_r][f,\Delta] =
  \C[\Omega_1,\dots,\Omega_r][f,\Delta,\bTheta]$.

  \noindent {\rm (ii)} $\D= \bigl(\bigoplus_{p \in \N} \tD
  f^p\bigr) \boplus \bigl(\bigoplus_{p \in \N^*}
  \tD\Delta^{p}\bigr)$.

  \noindent {\rm (iii)} For $k \in \Z$, set
  $$
  \D[k]=
  \begin{cases}
    \tD f^k \ &\text{if $k \ge 0$,} \\
    \tD \Delta^{-k} \ &\text{if $k < 0$;}
  \end{cases}
$$
then $\D[k] = f^k\tD$, ik $k \ge 0$, or $\Delta^{-k}\tD$, if
$k < 0$.

\noindent {\rm (iv)} $R = \rad(\D) = U =
\C[z,\delta,\theta]$.

\noindent {\rm (v)} $J = \Ker(\rad) = \sum_{i=0}^r \D
\Omega_i = \sum_{i=0}^r \Omega_i\D$.
\end{thm}

\begin{proof}
  Endow $\DV$, $\D$ and $\tD$ with the ``Bernstein
  filtration'', i.e.:
  $$
  \calF_p\DV = \sum_{i + j \le p} S^i(V^*)S^j(V), \quad
  \calF_p\D= \calF_p\DV \cap \D \supset \calF_p\tD=
  \calF_p\DV \cap \tD.
  $$
  Then, since $\tG$ and $G$ are reductive,
  $$
  \bbtS= \bigl[\gr_\calF \DV\bigr]^\tG = \bigl[S(V^*)
  \otimes S(V)\bigr]^\tG, \quad \bbS= \bigl[\gr_\calF
  \DV\bigr]^G = \bigl[S(V^*) \otimes S(V)\bigr]^G.
  $$
  Denote by $\sigma_j = \gr_{\calF}(E_j) \in
  \bigl[V(\tlambda_j) \otimes Y(\tlambda_j)\bigr]^\tG$ the
  principal symbol of $E_j$ for $\calF$. Then $\bbtS=
  \C[\sigma_0,\dots,\sigma_r]$, see for example~\cite{BR}.
  Recall that $E_0 = f\Delta$, hence $\sigma_0= ff^*$.  By
  Lemma~\ref{lem43} and Remark~\ref{rem46} we know that
  $S(V^*) = H(V^*) \otimes \C[f]$, $S(V) = H(V) \otimes
  \C[f^*]$, hence $\bbS= [H(V^*) \otimes H(V)]^G \otimes
  \C[f,f^*]$ (as vector spaces). Let $\gamma,\lambda \in
  \Gamma$; recall that the $G$-module $V(\gamma)$ is
  isomorphic to $E(p(\gamma))$ and that $p(\gamma) =
  p(\lambda)$ if and only if $\gamma = \lambda$,
  cf.~Lemma~\ref{lem42}. It follows that $[V(\gamma) \otimes
  Y(\gamma)]^G = [V(\gamma) \otimes Y(\gamma)]^\tG = \C
  E_\gamma$ and
  $$
  [H(V^*) \otimes H(V)]^G = \bigoplus_{\gamma \in \Gamma}\C
  E_\gamma \subset \bbtS =\C[\sigma_0,\dots,\sigma_r].
  $$
  Thus:
  $$
  \bbtS[f,f^*] \subset \bbS= [H(V^*) \otimes H(V)]^G \otimes
  \C[f,f^*] \subset \bbtS[f,f^*].
  $$
  % Notice that one can choose $\sigma_0 = ff^*$ and that
% $$
% \C[f,f^*] = \bigl(\bigoplus_{j \ge 0 \C[\sigma_0]f^i\bigr)
%   \boplus \bigl(\sum_{i >0} \C[\sigma_0] (f^*)^i}.
% $$
  Since the centre $C$ of $\tG$ acts trivially on $\bbtS$
  and via $\chi^j$, resp.~$\chi^{-j}$, on $f^j$,
  resp.~$(f^*)^j$, we obtain:
  \begin{equation*}
    \tag{$\star$}
    \bbS =  \bbtS[f,f^*] =\Bigl(\bigoplus_{j \ge 0 } \bbtS
    f^i\Bigr) 
    \boplus \Bigl(\bigoplus_{i >0} \bbtS (f^*)^i\Bigr) =
    \C[\sigma_1,\dots,\sigma_r] \otimes \C[f,f^*]. 
  \end{equation*}
  Then, by a filtration argument, one deduces that $\D=
  \sum_p \tD f^p + \sum_p \tD \Delta^p = \tD[f,\Delta] =
  \C[E_1,\dots,E_r][f,\Delta]=
  \C[\Omega_1,\dots,\Omega_r][f,\Delta,\bTheta]$ (recall
  that $\Omega_j = E_j - b_{E_j}(\bTheta)$). This proves
  (i).

  Observe that $\tD f^p$ and $f^p\tD$, resp.~$\tD\Delta^p$
  and $\Delta^p \tD$, are contained in the $\chi^p$-weight
  space, resp.~$\chi^{-p}$-weight space, for the action of
  $C$ on $\D$. This implies easily, as in~($\star$), that
  these one dimensional subspaces are equal to the
  corresponding weight spaces. This proves (ii) and (iii).

  Since $\Omega_i \in J = \Ker(\rad)$, we obtain $\rad(\D) =
  \C[\rad(f),\rad(\Delta), \rad(\bTheta)] =
  \C[z,\delta,\theta]$, hence (iv).

  (v) Let $P \in \D$ and write $P = \sum_k P_k$, $P_k \in
  \D[k]$ with $P_k = Q_k f^k$ or $Q'_k \Delta^{-k}$ and
  $Q_k, Q'_k \in \tD = \C[\Omega_0,\dots,\Omega_r,\bTheta]$.
  Set:
$$
Q_p = \sum_{i \ge 0} Q_{p,i}\bTheta^i, \quad Q_p = \sum_{i
  \ge 0} Q'_{p,i}\bTheta^i
$$
where $Q_{p,i}, Q'_{p,i} \in
\C[\Omega_0,\dots,\Omega_r]$. Since $Q_{p,i} \in Q_{p,i}(0)
+ \sum_j \Omega_j \tD$, $Q'_{p,i} \in Q'_{p,i}(0) + \sum_j
\Omega_j \tD$, $Q_{p,i}(0), Q'_{p,i} (0) \in \C$, we obtain
by applying $\rad$:
$$
\rad(P) = \sum_{k \ge 0}\Bigl(\sum_{i \ge 0} Q_{k,i}(0)
\theta^i\Bigr) z^k + \sum_{p > 0}\Bigl(\sum_{i \ge 0}
Q'_{p,i}(0) \theta^i\Bigr) \delta^p.
$$
Recall from \S\ref{ssec21} that there exists a filtration on
$R$ such that $\gr(R)$ is isomorphic to $\C[X,Y,S]/(XY -
S^n)$ where $\gr(z) \equiv X, \gr(\delta) \equiv Y,
\gr(\theta) \equiv S$ (up to some scalar). This implies
easily that $R= \bigl(\bigoplus_{k \ge 0} \C[\theta]
z^k\bigr) \boplus \bigl(\bigoplus_{p > 0} \C[\theta]
\delta^p\bigr)$. Now suppose that $P \in J$, then $\rad(P) =
0 = \sum_{k \ge 0}\bigl(\sum_{i \ge 0} Q_{k,i}(0)
\theta^i\bigr) z^k + \sum_{p > 0}\bigl(\sum_{i \ge 0}
Q'_{p,i}(0) \theta^i\bigr) \delta^p$ forces $\sum_{i \ge 0}
Q_{k,i}(0) \theta^i = \sum_{i \ge 0} Q'_{p,i}(0) \theta^i =
0$, hence $Q_{k,i}(0) = Q'_{p,i}(0) = 0$ for all
$k,p,i$. This shows that $Q_k, Q'_p \in \sum_{i=0}^r
\Omega_0\tD + \cdots + \Omega_r \tD$ and therefore $P \in
\sum_{i=0}^r \Omega_i \D$. Writing $P_k \in f^k \tD$ or
$\Delta^{-k} \tD$ yields $P \in \sum_{i=0}^r \tD \Omega_0 +
\cdots + \tD \Omega_r$.
\end{proof}

\begin{rem}
  \label{rem411}
  (1) In the case where $\tGV$ is irreducible we have
  noticed that $\Omega_r = 0$, see~\eqref{eq46}, thus
  $$
  J = \sum_{i=0}^{r-1} \D \Omega_i = \sum_{i=0}^{r-1}
  \Omega_i\D.
  $$
  
  (2) From $[\bTheta,f^k] = kf^k$ and $[\bTheta,\Delta^k] =
  -k \Delta^k $ we can deduce that
  $$
  \D[k] = \{D \in \D : [\bTheta,D] = kD\}.
  $$
\end{rem}

\subsection{A Howe duality}
\label{ssec43}
The Howe duality for the Weil representation gives a
bijection between irreducible finite dimensional
representations of $\SO(n)$ and irreducible lowest weight
$U(\fsl(2))$-modules. Algebraically, this result corresponds
to the case $(\tG = \SO(n) \times \C^* : V = \C^n)$: here
the Lie subalgebra of $\DV$ generated by $f$ (quadratic
form) and $\Delta$ (Laplacian) is isomorphic to $\fsl(2)$.
More precisely we have the following result.  Let $\calA
\cong \fsl(2)$ be this Lie algebra, denote by $H_d \subset
H(V^*)$ the space of harmonic polynomials of degree $d$,
i.e.:
$$
H_d = \{q \in S^d(V^*) : \Delta(q) = 0\}.
$$
Each $H_d$ is an irreducible $\SO(n)$-module and the $\calA
\times \SO(n)$-module $\CV$ decomposes as
$$
\CV = \bigoplus_d X(d+n/2) \otimes H_d
$$
where $X(d+n/2)$ is the irreducible lowest $\fsl(2)$-module
of lowest weight $d+n/2$.

This kind of duality has been extended by H.~Rubenthaler
\cite{Rub1} to a class of PHV, the so-called {\em
  commutative parabolic} PHV. They are associated to short
gradings $\fs = \fs_{-1} \oplus \fs_0 \oplus \fs_1$ on
simple Lie algebras.  The commutative parabolic PHV are
irreducible, MF and satisfy $\dim V \qmod G =1$, thus
Hypothesis~\ref{hypB} holds.  We want to generalize the Howe
duality to the more general class of representations $\tGV$
satisfying only Hypothesis~\ref{hypB}. We therefore fix a
representation $\tGV$ satisfying this hypothesis,
see~\S\ref{ssec42}. We have indicated in the last column of
the table of Appendix~\ref{A1} the irreducible MF
representations $\tGV$ which are of commutative parabolic
type.

Let
\begin{equation}
  \label{eq47}
  \calA = \Lie\langle f, \Delta\rangle \subset \bigl(\DV,
  [\phantom{s},\phantom{s}]\bigr)
\end{equation}
be the Lie subalgebra generated by $f,\Delta$.  Notice that
$\calA \subset \D$. Let $\tgamma = \sum_{j=0}^r a_j
\tlambda_j$. Recall that $\Vtgamma = U(\g).h^\tgamma$; we
then set:
$$
\asf = (a_0,\dots,a_r), \quad h^\tgamma = h^\asf =
f^{a_0}h_1^{a_1}\cdots h_r^{a_r}, \quad \Vtgamma = V(\asf) =
V(a_0,\dots,a_r).
$$
Let $b \in \N$, by Lemma~\ref{lem44}(b), the operator
$\Delta^b$ acts as follows: $\Delta^b(V(\asf))= 0$ if $a_0 <
b$, $\Delta^b : V(a_0,\dots,a_r) \isomto V(a_0
-b,a_1,\dots,a_r)$ if $b \le a_0$, and in the latter case
$\Delta^b(h^\asf) \in \C^* f^{a_0-b}h_1^{a_1}\cdots
h_r^{a_r}$ is a highest weight vector in $V(a_0
-b,\dots,a_r)$. Obviously, the multiplication by $f^b$ gives
an isomorphism $f^b : V(a_0,\dots,a_r) \isomto V(a_0
+b,a_1,\dots,a_r)$ of $\tG$-modules. Define:
\begin{equation}
  \label{eq48}
  \calA_j = \{D \in \calA : D(V(a_0,\dots,a_r)) \subset
  V(a_0+j,a_1,\dots,a_r) \ \text{for all $\asf \in
    \N^{r+1}$}\}. 
\end{equation}
It is clear that $\bigoplus_{j \in \Z} \calA_j \subset
\calA$ and, by the previous remarks, $f^b \in \calA_b,
\Delta^b \in \calA_{-b}$.

\begin{remarks}
  1) It is difficult to compute in the Lie algebra $\calA$
  because $[\Delta,f] = \psi(-\bTheta) + Q$ for some $Q\in
  \Ker(\rad)$ which is not easy to calculate (recall that
  here $\psi(s) = b(-s) -b(-s-1)$, cf.~\eqref{eq27}). For
  example when $\tGV = (\GL(n) \times \SL(n) : \Mat_n(\C))$
  P.~Nang \cite{Nang3} has shown that (up to some scalar):
  $Q =
  \mathrm{trace}(\mathbf{x}^{\#}\boldsymbol{\partial}^{\#})$
  where $\mathbf{x}^{\#}$,
  resp.~$\boldsymbol{\partial}^{\#}$, is the adjoint matrix
  of $\mathbf{x}=[x_{ij}]_{ij}$,
  resp.~$\boldsymbol{\partial} = [\partial_{x_{ij}}]_{ij}$.
  (See also \cite[Proposition~7]{Nang5} for the case
  $(\GL(2m) : \Wedge^2\C^{2m})$.)
  
  \noindent 2) When $n = \deg f =2$ one has $\calA \cong
  \fsl(2)$, thus $\dim_\C \calA =3$.
\end{remarks}

The assertion~(2) of the next proposition should be compared
with \cite[Th\'e\-or\`e\-me~3.1]{Rub1}.

\begin{prop}
  \label{prop412} {\rm (1)} One has $\calA_k = \calA \cap
  \D[k]$, $\calA = \bigoplus_{k \in \Z} \calA_k$. The Lie
  subalgebra $\calA_0$ is abelian.

  \noindent {\rm (2)} Suppose $n \ge 3$, then $\dim_\C
  \calA_k = \infty$ for all $k \in \Z$.
\end{prop}

\begin{proof}
  (1) Since $f, \Delta \in \calA$ the relations $[\bTheta,f]
  =f$ and $[\bTheta, \Delta] = -\Delta$ show that
  $\ad(\bTheta) : \D \to \D$ induces an endomorphism of
  $\calA$.  From $\calA \subset \D$ we deduce that $\calA =
  \bigoplus_{k \in \Z} \calA \cap \D[k]$. let $P \in \D[k]$,
  $P= f^k D$ or $\Delta^{-k}D$ be an element of $\calA \cap
  \D[k]$. Then, using Remark~\ref{rem48}(3), we see that $P
  \in \calA_k$. The desired equalities follow easily.  Since
  $\calA_0 \subset \tD$ and $\tD$ is a commutative algebra,
  cf.~Theorem~\ref{thm46}, $\calA_0$ is abelian.
  
  (2) We claim that $\rad(\calA_j) = \calL_j$, where
  $\calL_j$ is defined as in \S\ref{ssec22}, i.e.: $\calL =
  \Lie\langle f,\Delta \rangle$, $\calL_i = \{u \in \calL :
  u(z^{m}) \in \C z^{m+j} \ \text{for all $m \in \N$}\}$
  (with the convention that $\C z^{m+j} = (0)$ when $m+j <
  0$). Note first that $\rad(\calA) = \Lie\langle \rad(f),
  \rad(\Delta)\rangle= \calL$. Let $D \in \calA_j$, then
  $\rad(D)(z^m) = \psi(D(f^m))$. Observe that $f^m \in
  V(m,0,\dots,0)$, hence $D(f^m) \in V(m+j,0,\dots,0)$,
  which is $(0)$ is $m+j <0$ and $\C f^{m+j}$ if $m + j \ge
  0$. Thus $\rad(D) \in \C z^{m+j}$ and $\rad(\calA_j)
  \subset \calL_j$.  It follows that $\rad(\calA) = \sum_j
  \rad(\calA_j) \subset \boplus_j \calL_j \subset \calL =
  \rad(\calA)$. Therefore $\rad(\calA_j) = \calL_j$ for all
  $j$ (and $\calL = \boplus_j \calL_j$). Now,
  Proposition~\ref{prop28} yields the desired assertion.
\end{proof}

Set $\calA_{+} = \bigoplus_{k >0} \calA_k$, $\calA_{-} =
\bigoplus_{k < 0} \calA_k$. We then have a triangular
decomposition $\calA = \calA_- \boplus \calA_0 \boplus
\calA_+$ which enables us to introduce the notion of a
lowest weight $\calA$-module, see~\cite{Rub1}. As usual in
this situation the weights will be elements of the linear
dual space $\calA_0^*$ of the abelian Lie algebra $\calA_0$.

\begin{defn}
  \label{def413}
  The $\calA$-module $X$ is a lowest weight module if there
  exist $x \in X$ and $\vphi \in \calA_0^*$ such that: $X =
  U(\calA).x$, $\calA_-.x = 0$, $a.x = \vphi(a)x$ for all $a
  \in \calA_0$.
\end{defn}

The next theorem generalizes \cite[Proposition~4.2]{Rub1};
it gives a Howe duality for MF representations with a one
dimensional quotient (see also \cite[Corollary~4.5.17]{GW}).
Recall that $H(V^*)= \bigoplus_{\gamma \in \Gamma}
V(\gamma)$, where $V(\gamma) \cong E(p(\gamma))$, is equal
to the space of harmonic elements, i.e.~$\{\vphi \in \CV :
\Delta_0(\vphi)= \cdots = \Delta_m(\vphi)=0\}$,
cf.~\eqref{eq42} and Proposition~\ref{prop45}.

\begin{thm}
  \label{thm414}
  Let $\tGV$ satisfying Hypothesis~\ref{hypB}.  The $\calA
  \times \g$-module $\CV$ decomposes as $\CV \cong
  \bigoplus_{\gamma \in \Gamma} X(\gamma) \otimes
  E(p(\gamma))$, where $X(\gamma) = U(\calA).h^{\gamma}$ is
  an irreducible lowest weight $\calA$-module. Moreover:
  $X(\gamma) \cong X(\gamma') \iff \gamma = \gamma'$.
\end{thm}

\begin{proof}
  Recall that $\tGamma = \N \tlambda_0 \boplus \Gamma$,
  $\Gamma = \boplus_{i=1}^r \N \tlambda_i$. Let $\gamma =
  \sum_{i=1}^r a_i \tlambda_i \in \Gamma$; set $P_\gamma =
  \gamma+ \N \tlambda_0$. Then $h^{\gamma} = h_1^{a_1}\cdots
  h_r^{a_r}$ and $V(\gamma)= U(\fg).h^{\gamma} \cong
  E(p(\gamma))$, see Lemma~\ref{lem42}. From
  $f^bh^{\gamma}=h^{(b,a_1,\dots,a_r)}$ and
  $\Delta^b(f^{a_0}h_1^{a_1}\cdots h_r^{a_r}) \in \C^*
  f^{a_0-b}h_1^{a_1}\cdots h_r^{a_r}$ when $a_0 \ge b$, and
  $0$ when $a_0 < b$, we get that $X(\gamma)=
  U(\calA).h^{\gamma} = \bigoplus_{\mu \in P_\gamma} \C
  h^\mu$ is an irreducible $U(\calA)$-module. Let $D \in
  \calA_{-j} = \calA \cap\D[-j]$, $j > 0$, and write $D= D'
  \Delta^j$, $D' \in \tD$.  Then $D(h^{\gamma}) = D'
  \Delta^j(h^{\gamma})$ and we have noticed that
  $\Delta^j(h^{\gamma}) =0$, thus $\calA_-.h^{\gamma} = 0$.
  When $D \in \calA_0 = \calA \cap \tD$,
  Remark~\ref{rem48}(3) gives that $D(h^{\gamma}) =
  \vphi(D)h^{\gamma} \in \C h^{\gamma}$. Since it is obvious
  that $\vphi \in \calA_0^*$, $X(\gamma)$ is an irreducible
  lowest weight module.
  
  By \cite[Theorem~4.5.16]{GW} we know that the $\D \times
  \g$-module $\CV$ has the following decomposition:
  $$
  \CV = \bigoplus_{\lambda \in \Lambda^+}
  \Hom_\g(E(\lambda), \CV) \otimes_\C E(\lambda)
  $$
  where $\Hom_\g(E(\lambda), \CV)$ is either $(0)$ or a
  simple $\D$-module (the action being given by $D(\phi)(x)
  = D(\phi(x))$ for all $\phi \in \Hom_\g(E(\lambda), \CV),
  x \in E(\lambda)$). Since the $G$-module $V(\tgamma)$,
  $\tgamma \in \tGamma$, is isomorphic to $E(p(\tgamma))$,
  we obtain that $\Hom_\g(E(\lambda), \CV) =(0)$ if $\lambda
  \notin p(\Gamma) = p(\tGamma)$. If $\lambda \in p(\Gamma)$
  the non zero elements of this $\D$-module identify with
  the $\g$-highest weight vectors of weight $\lambda$ in
  $\CV$ through the map $\phi \mapsto \phi(v_\lambda)$,
  where $v_\lambda$ is a highest weight vector in
  $E(\lambda)$.
  % If $\tgamma = k \tlambda_0 + \gamma$, $\gamma \in
  % \Gamma$ the $\g$-highest weight of $V(\tgamma)$ is
  % $p(\gamma)= p(\tgamma)$, therefore the $\g$-highest
  % weight vectors of weight $\lambda = p(\gammma)$ in $\CV$
  % are the $h^\tgamma$ with $\tgamma = k \tlambda_0 +
  % \gamma$, $p(\gamma)= \lambda$.
  It is easily seen that the $\g$-highest weight vectors of
  weight $\lambda$ in $\CV$ are the $h^\tgamma$ with
  $\tgamma = k \tlambda_0 + \gamma'$, $p(\gamma') =
  \lambda$.  Recall that for $\gamma,\gamma' \in \Gamma$,
  $p(\gamma)=p(\gamma') \iff \gamma=\gamma'$; therefore
  these $\g$-highest weight vectors are the $h^\tgamma$ with
  $\tgamma \in P_\gamma = \gamma + \N \tlambda_0$, where
  $\gamma \in \Gamma$ is such that $p(\gamma)=\lambda$. From
  the previous paragraph we then obtained that
  $\Hom_\g(E(\lambda), \CV) \cong X(\gamma)$ as
  $\calA$-module. The last assertion follows from
  \cite[Theorem 4.5.12]{GW}.
\end{proof}

%%%%%%%%%%%%%%%%%%%%%%%%
\section{$D$-modules on some PHV}
\label{sec5}
%%%%%%%%%%%%%%%%%%%%%%

In this section we continue with a representation $\tGV$ of
the connected reductive group $\tG$ as in \S\ref{ssec31}.

\subsection{Representations of Capelli type}
\label{ssec51}
Let
$$
\tau : \tilde{\g} = \Lie(\tG) \lto \DV
$$
be the differential of the $\tG$-action. The elements
$\tau(\xi)$ are linear derivations on $\CV$ given by:
$$
\tau(\xi)(\vphi)(v) = \frac{d}{dt}_{\mid t
  =0}(\Exp^{t\xi}.\vphi)(v)= \frac{d}{dt}_{\mid t
  =0}\vphi(\Exp^{-t\xi}.v),
$$
for all $\vphi \in \CV, v \in V$.  They are homogeneous of
degree zero in the sense that $[\Theta,\tau(\xi)]=0$.  The
map $\tau$ yields a homomorphism $\tau : U(\tilde{\g}) \to
\DV$.

Recall that the group $\tG$ acts naturally on $\DV$; the
differential of this action is given by $D \mapsto
[\tau(\xi),D]$, $\xi \in \tfg$, $D \in \DV$. Therefore, a
subspace $I \subset \DV$ is stable under $\tG$, resp.~$G$,
if and only if $[\tau(\tfg,I] \subset I$, resp.~$[\tau(\g,I]
\subset I$. It is then clear that $\tD= \DVtG = \{D \in \DV
: [\tau(\tfg),D] = 0\} \subset \D = \DVG = \{D \in \DV :
[\tau(\g),D] = 0\}$. In particular, if $Z(\tfg) =
U(\tfg)^{\tG}$ is the centre of $U(\tilde{\g})$, then
$\tau(Z(\tilde{\g})) \subset \tD$.

Following \cite[(10.3)]{HU} we make the following
definition:

\begin{defn}
  \label{capelli}
  We say that $\tGV$ is of Capelli type if:
  \begin{itemize}
  \item $\tGV$ is irreducible and multiplicity free;
    % satisfies the hypothesis~\ref{hypB}, i.e.~$\tGV$ and
    % such that $\dim V \qmod G =1$;
  \item $\tau(Z(\tilde{\g})) = \tD$.
  \end{itemize}
\end{defn}

\begin{rems}
  \label{rems52}
  (1) By Howe and Umeda \cite{HU}, in the list of $\tGV$
  which are irreducible and MF we have: three among the
  thirteen cases are not of Capelli type; two among the ten
  cases such that $\dim V \qmod G = 1$ are not of Capelli
  type.  (Thus we are interested in eight cases of the table
  in Appendix~\ref{A1}.)
  
  (2) This definition originates in the case $(\tG = \GL(n)
  \times \SL(n) : V = \Mat_n(\C))$ where the writing of $E_0
  = f\Delta = \det(x_{ij})\det(\partial_{ij})$ as an element
  of $\tau(Z(\tilde{\g}))$ is the ``classical'' Capelli
  identity.
\end{rems}

Recall the morphism $\rad : \DVG \to \calD(V\qmod G)$.  By
definition $\tau(\g)(\CVG) = 0$, thus one always has:
$$
[\DV \tau(\g)]^G \subset J = \Ker(\rad).
$$
When $\tGV$ satisfies Hypothesis~\ref{hypB}, we have
computed in Theorem~\ref{thm410} the ideal $J \subset \D$:
$$
J = \sum_{i=0}^r \D \Omega_i = \sum_{i=0}^r \Omega_i\D
$$
where the $\Omega_i$'s are defined in~\eqref{eq45}. When
$\tGV$ is irreducible we observed in Remark~\ref{rem411}(1)
that we can number these operators so that $\Omega_r =0$,
hence $J = \sum_{i=0}^{r-1} \D \Omega_i$; the next
proposition gives a more useful description if, moreover,
$\tGV$ is of Capelli type, i.e.~one of the eight cases
mentioned in Remark~\ref{rems52}(1).

\begin{prop}
  \label{prop53}
  Let $\tGV$ be of Capelli type and such that $\dim V \qmod
  G = 1$.  Then:
  $$
  \Ker(\rad) = [\DV\tau(\g)]^G.
  $$
\end{prop}

\begin{proof}
  In the irreducible case the centre $C$ of $\tG$ acts by
  scalars on $V$ and we may assume that: $\tG = GC$ with $C
  = \C^*$. Write $\tfg = \g \boplus \fc$, $\fc =\Lie(C) = \C
  \zeta$.  Since $\CVG = \C[f]$ and $f$ is not
  $\tG$-invariant one can also suppose that $\tau(\zeta) =
  \bTheta = \frac{1}{n}\Theta$.
  
  Write $Z(\tfg)= Z(\g)[\zeta]$ and $Z(\g) = \C \boplus
  Z_+(\g)$ where $Z_+(\g)= [U(\g)\g]^G$. The previous
  paragraph implies that $\tau(Z(\tfg))= \C[\bTheta] +
  \tau(Z_+(\g))[\bTheta]$ with $\tau(Z_+(\g)) =
  [U(\tau(\g))\tau(\g)]^G \subset \DVtauG$. As recalled
  above we already know that $J = \sum_{i=0}^{r-1} \D
  \Omega_i$, $\Omega_i \in \tD$, $0 \le i \le r-1$. By
  hypothesis $\tau(Z(\tilde{\g})) = \tD$, thus we can write
  each $\Omega_j$ as follows:
  $$
  \Omega_j = \sum_{k \ge 0} \bTheta^k P_{j,k}, \quad P_{j,k}
  = p_{j,k} + P'_{j,k}, \ p_{j,k} \in \C, \ P'_{j,k} \in
  \tau(Z_+(\g)).
  $$
  Recall that $[\DV \tau(\g)]^G \subset J$; thus we have
  $\rad(P'_{j,k}) = 0$ and we obtain: $\rad(\Omega_j) =
  \sum_{k \ge 0} \theta^k p_{j,k}=0$ in $R
  =\C[z,\delta,\theta]$.  Therefore $p_{j,k} =0$ for all $k
  \ge 0$, which gives $\Omega_j = \sum_{k \ge 0} \bTheta^k
  P'_{j,k} \in \DVtauG$ and $J \subset \DVtauG$.
\end{proof}

\subsection{Application to $D$-modules}
\label{ssec52} {\em In this subsection we assume that $\tGV$
  satisfies Hypothesis~\ref{hypB}, hence $\tGV$ is MF,
  $\CVG= \C[f]$, $f \notin \C[V]^{\tG}$.}

Recall from Theorem~\ref{thm410} that $\D = \bigoplus_{k \in
  Z} \D[k]$. We have seen (Lemma~\ref{lem49}) that if $D \in
\tD$ there exists a polynomial $b_D(s)$ such that $D =
b_D(\bTheta) + D_1$, $D_1 \in J = \Ker(\rad)$.

Fix $P \in \D[k]$ and write $P = DQ_k$, $Q_k = f^k$ if $k
\ge 0$, $Q_k = \Delta^{-k}$ if $k <0$, $D \in \tD$. Then:
\begin{equation}
  \label{eq51}
  P = b_D(\bTheta)Q_k + P_1, \quad P_1 = D_1Q_k \in J.
\end{equation}
Observe that, since $\Delta(f^m) = b^*(m)f^{m-1}$,
$$
P(f^m) =
\begin{cases}
  b_D(m+k)f^{m+k}  & \  \text{if $k \ge 0$;}\\
  b^*(m)b^*(m-1)\cdots b^*(m+k+1)b_D(m+k) f^{m+k} & \
  \text{if $k < 0$.}
\end{cases}
$$
Therefore if we set $a_P(s) = b_D(s+k)$ if $k \ge 0$, or
$a_P(s) = b_D(s+k)\prod_{j=0}^{k+1}b^*(s+j)$ if $k < 0$,
then $\deg a_P = \deg b_P$ or $\deg b_P + nk$, and $P(f^m) =
a_P(m) f^{m+k}$. Notice that $a_{Q_k}(s) =1$ if $k \ge 0$,
$a_{Q_k}(s) =\prod_{j=0}^{k+1}b^*(s+j)$ if $k < 0$, thus
$a_P(s) = b_D(s+k)a_{Q_k}(s)$.

% \begin{remark}
When $\tGV = (\GL(n) : S^2\C^n)$, similar results were
obtained by Masakazu Muro in \cite[Proposition~3.8]{Mu4},
where our polynomial $a_P(s)$ is denoted by $b_P(s)$.
% \end{remark}

\begin{defn}
  \label{defn54}
  Let $P \in \D[k]$ be as above and define the $D$-module
  associated to $P$ by:
  $$
  \calM_P = \DV \big/(\DV\tau(\g) + \DV P) = \DV/I_P,
  $$
  where $I_P= \DV\tau(\g) + \DV P$.
\end{defn}

\begin{remark}
  Let $I \subset \DV$ be a left ideal containing
  $\DV\tau(\g)$. Since the condition $[\tau(\g),I] \subset
  I$ is satisfied, the group $G$ acts naturally on $I$ and
  therefore on $M=\DV/I$. Furthermore, the differential of
  this action is given by the left multiplication on $M$ by
  $\tau(\xi)$, $\xi \in \g$. It follows, see \cite[\S II.2,
  Theorem]{Hot2}, that $M$ is a $G$-equivariant $D$-module
  on $V$ (cf.~\cite{Hot2} for the definition). This is in
  particular true for $\calM_P$.
\end{remark}

Following \cite{Mu1,Mu2,Mu3,Mu4} we want to study the
solutions of the differential system associated to
$\calM_P$. We first need to study the characteristic variety
on $\cMP$.

Recall that $\DV$ is filtered by the order of differential
operators, see~\cite{Bo, Hot1}, and that the associated
graded ring of $\DV$ identifies with $\C[T^*V]$, where $T^*V
= V \times V^*$ is the cotangent bundle of $V$.  If $u \in
\DV$ we denote its order by $\ord u$ and its principal
symbol by $\sigma(u) \in \C[T^*V] = S(V^*) \otimes S(V)$.
Let $M$ be a finitely generated $\DV$-module, then one can
endow $M$ with a good filtration and one defines the
characteristic variety of $M$, denoted by $\ch M$, as being
the support in $T^*M$ of the associated graded module (see,
e.g., \cite[\S I.3]{Hot1}). When $M = \DV/ I$, $\ch M
\subset T^*M$ is the variety of zeroes of the symbols
$\sigma(u)$, $u \in I$.  Recall that if $\dim \ch M \le \dim
V$ the $D$-module $M$ is called holonomic (one always have
$\dim \ch M \ge \dim V$ if $M \ne (0)$).

\begin{remark}
  Let $v \in V, v^* \in V^*, \xi \in \tfg$. One has:
  $$
  \sigma(\Theta)(v,v^*) = \dual{v^*}{v}, \quad
  \sigma(\tau(\xi))(v,v^*) = -\dual{\xi.v^*}{v} =
  \dual{v^*}{\xi.v}.
  $$
\end{remark}

\begin{lem}
  \label{lem540}
  Let $k \in \Z$ and $P= DQ_k$ be as above.
  % Set $Q_k^* = f^{-k}$ if $k < 0$ and $Q_k^* = \Delta^{k}$
  % if $k \ge 0$.

  \noindent {\rm (a)} There exists $Q_k' \in J$ and $q_k(s)
  \in \C[s]$ such that $Q_kQ_{-k} = q_k(\bTheta) + Q_k'$,
  $\ord(Q_kQ_{-k}) = \ord(Q_{-k}Q_{k})= \deg q_k = |k|n$.

  \noindent {\rm (b)} Write $P = b_D(\bTheta)Q_k + P_1$ as
  in~\eqref{eq51} and set $P_0 = b_D(\bTheta)Q_k$. Then
  $P_0Q_{-k} = b_D(\bTheta)q_k(\bTheta) + P_2$ with $P_2 \in
  J$ and $\ord P_2 \le \ord(P_0Q_{-k}) = \deg (b_Dq_k)$.
\end{lem}

\begin{proof}
  (a) Let $m \in \N$. Then:
  $$
  Q_kQ_{-k}(f^m)= a_{Q_{-k}}(m)Q_k(f^{m-k})=
  a_{Q_{-k}}(m)a_{Q_k}(m-k)f^m.
  $$
  Set $q_k(s) = a_{Q_{-k}}(s)a_{Q_k}(s-k) \in \C[s]$. The
  previous computation yields $Q_kQ_{-k} = q_k(\bTheta) +
  Q_k'$ with $Q_k' \in J$. Since $\deg a_{Q_k} =1$ or $-nk$
  (if $k \ge 0$ or $<0$) we get that $\deg q_k = |k|n$. On
  the other hand, $Q_kQ_{-k} = f^k\Delta^{-k}$ or
  $\Delta^{k}f^{-k}$ has order $\ord \Delta^{|k|}$,
  i.e.~$|k|n = \deg q_k$. This implies in particular that
  $\ord Q_k' \le \ord(Q_kQ_{-k})$.
  
  (b) From (a) we obtain $P_0Q_{-k} =
  b_D(\bTheta)q_k(\bTheta) + P_2$, $P_2 =
  b_D(\bTheta)Q'_k\in J$. One has: $\ord(P_0Q_{-k})= \ord
  P_0 + \ord Q_{-k} = \deg b_D + \ord Q_k + \ord Q_{-k} =
  \deg b_D + \ord(Q_kQ_{-k})$. Therefore, $\ord P_2 = \deg
  b_D + \ord Q'_k \le \deg b_D + \ord(Q_kQ_{-k}) =
  \ord(P_0Q_{-k})$, as desired.
\end{proof}

As in \cite[Section~3]{Pa2} define the commuting varieties
of $\tGV$ and $\GV$ by:
$$
\tcV =\bigl\{(v,v^*) \in T^*V : \dual{v^*}{\tilde{\g}.v} =0
\bigr\}, \quad \cV =\bigl\{(v,v^*) \in T^*V :
\dual{v^*}{\g.x} =0 \bigr\}.
$$

Recall that $\tGV$ is MF; this implies \cite{Kac} that $\tG$
has finitely many orbits in $V$, denoted by $O_1,\dots,O_t$.
Set $T^*_{O_i}V= \{(v,v^*) \in T^*V : v \in O_i, \,
\dual{v^*}{\tfg.v} = 0\}$. By \cite{Pya}, see
also~\cite[Theorem~3.2]{Pa2}, we have the following result:

\begin{thm}
  \label{Pya}
  The irreducible components of $\tcV$ are the closures of
  the conormal bundles of the orbits, i.e.~the $\calC_i =
  \overline{T_{O_i}^*V}$. In particular, $\tcV$ is
  equidimensional of dimension $\dim V$.
\end{thm}

\begin{rem}
  \label{rem55}
  Set $\cV' = \bigl\{(v,v^*) \in T^*V : \sigma(u)(v,v^*) = 0
  \; \text{for all $u \in \DV\tau(\g)$}\}$.  Thus $\cV'$ is
  the characteristic variety of the $\DV$-module
  \begin{equation}
    \label{eq52}
    \cN = \DV \big/ \DV\tau(\g)
  \end{equation}
  Let $P \in \D[k]$.  Then we clearly have:
  $$
  \tcV \subset \cV, \quad \ch \cMP \subset \cV' \subset \cV.
  $$
\end{rem}

{\em We will now assume that $\tGV$ is irreducible.}  This
means that $\tGV$ is one of the cases (1) to (10) in the
table of Appendix~\ref{A1}.  We may assume here that $\tG=
GC$, $C \cong \C^*$.  We then write $\tfg = \g \boplus
\C\zeta$, $\fc= \Lie(C) = \C \zeta$ where $\zeta$ is chosen
such that $\tau(\zeta) = \bTheta$ (i.e.~$\zeta$ acts as
$\frac{1}{n} \id_V$ on $V$).

\begin{cor}
  \label{cor56}
  Under the previous hypothesis:
  
  -- $\dim \cV = \dim V +1$;
  
  -- the module $\cN$ is not holonomic, i.e.~$\dim \cV' =
  \dim V +1$.
\end{cor}

\begin{proof}
  By \cite[Theorem~1]{Kac} the representation $\GV$ is
  visible, i.e.~$\{v \in V : f(v) = 0\}$ contains a finite
  number of $G$-orbits.  Then, since $\C(V)^G = \C(f)$ and
  $f$ is non constant, \cite[Theorems~2.3 \&~3.1,
  Corollary~2.5]{Pa1} yield $\dim \cV = \dim V + 1$.
  
  Recall that if $M$ is any $\DV$-module one can identify
  $\Hom_\DV(\cN,M)$ with the space $\{x \in M : \tau(\g).x =
  0\}$. In particular, $\Hom_\DV(\cN,\CV) \equiv \CVG =
  \C[f]$. If $\cN$ were holonomic we would have $\dim_\C
  \Hom_\DV(\cN,\CV) < \infty$, cf.~\cite[Theorem~9.5.5]{MR},
  which is absurd.
\end{proof}

Assume that $\tGV$ is of Capelli type and let $P =
b_D(\bTheta)Q_k + P_1$, $P_1 \in J$, as in~\eqref{eq51}. By
Proposition~\ref{prop53}, $I_P= \DV\tau(\g) + \DV
b_D(\bTheta)Q_k$. Therefore:
\begin{equation}
  \label{eq53}
  \calM(b_D,k) = \calM_P = \DV \big/(\DV\tau(\g) + \DV
  b_D(\bTheta) Q_k)
\end{equation}
depends only on the polynomial $b_D(s)$ and the integer $k$.
We need to know in which case $\calM(b_D,k)$ is holonomic.

\begin{thm}
  \label{thm57}
  Assume that $\tGV$ is of Capelli type and let $P=
  b_D(\bTheta)Q_k + P_1$ with $P_1 \in \Ker(\rad)$. The
  following are equivalent:
  $$
  {\rm (i)} \ b_D(s) \ne 0 \, ; \ \; {\rm (ii)} \
  \calM(b_D,k) = \calM_P \ \text{is holonomic.}
  $$
  In this case $\ch \calM_P \subset \tcV$ is a union of
  $\calC_i$'s.
\end{thm}

\begin{proof}
  Suppose that $b_D = 0$, then $\calM(b_D,k) = \calN$ is not
  holonomic. Conversely, suppose $b_D \ne 0$. We are going
  to show that $\calM_P \subset \tcV$, then
  Theorem~\ref{Pya} will give the result.
  
  Set $\alpha= \sigma(\bTheta) = \sigma(\tau(\zeta)) \in
  \C[T^*V]$, hence $\alpha(v,v^*) =
  \frac{1}{n}\dual{v^*}{v}$. Since $\tfg = \g \boplus \C
  \zeta$, we have $\tcV = \cV \cap \alpha^{-1}(0)$. Using
  the notation of Lemma~\ref{lem540} we set $P_0 =
  b_D(\bTheta) Q_k$, $P_0 Q_{-k} = b_D(\bTheta)q_k(\bTheta)
  + P_2$ with $P_2 \in J$. Notice that $\cMP = \calM_{P_0}$.
  Since $b_D \ne 0$, we can write $h(s) = b_D(s)q_k(s) = h_d
  s^d + h_{d-1} s^{d-1} + \cdots$, $d = \deg h(s) = \deg
  b_D(s) + \deg q_k(s) = \deg b_D(s)+ |k|n \ge 0$.  Recall
  that $\ord P_2 \le \ord P_0Q_k = d$, therefore
  $\sigma(P_0Q_k) = \sigma(h(\bTheta))$ or
  $\sigma(h(\bTheta)) + \sigma(P_2)$ (depending on $\ord P_2
  < d$ or $\ord P_2= d$).
  
  If $d=0$, we get $k=0$, $b_D \in \C^*$, thus $\cMP= (0)$
  and the claim is obvious. Suppose $d \ge 1$ and let
  $(v,v^*) \in \ch \cMP \subset \ch \calN = \cV'$. From $P_2
  \in J=\DVtauG \subset \DV\tau(\g)$ we know that
  $\sigma(P_2)(v,v^*) = 0$. Therefore, $\sigma(P_0)(v,v^*) =
  0$ implies
  \begin{align*}
    0 & =\sigma(P_0)(v,v^*)\sigma(Q_{-k})(v,v^*)=
    \sigma(P_0Q_{-k})(v,v^*) = \sigma(h(\bTheta))(v,v^*)
    \\
    & = h_d \sigma(\bTheta)^d(v,v^*) = h_d \alpha(v,v^*)^d.
  \end{align*}
  Hence $\alpha(v,v^*) =0$ and this proves $(v,v^*) \in \cV'
  \cap \alpha^{-1}(0) \subset \tcV$, as desired.
\end{proof}

\begin{rem}
  \label{remuro}
  The previous result generalizes the main step in the proof
  of \cite[Theorem~4.1]{Mu4} (see also \cite[Theorem~6.1 and
  Remark~6.1]{Mu4}). Indeed the homogeneous elements of
  degree $kn$ in \cite{Mu4} are the elements of $\D[k]$ and
  \cite[Lemma~4.1]{Mu4} shows, when $\tGV = (\GL(n) :
  S^2\C^n)$, that $\calM_P$ is holonomic when $P \ne 0$ and
  $\deg a_P = \ord P$.  Since $a_P(s) = b_D(s+k)a_{Q_k}(s)
  \ne 0 \iff b_D(s) \ne 0$, Theorem~\ref{thm57} ensures that
  a more general result holds for representations of Capelli
  type; notice that the example given \cite[Remark~4.1]{Mu4}
  is $P=\Omega_0 = f\Delta - b^*(\bTheta) \in J$, hence
  $\calM_P = \calN$ is not holonomic.  Observe also that
  Theorem~\ref{thm57} is proved in
  \cite[Proposition~2.1]{Mu1} in the case $(b_D =1, k \ge
  0)$.
\end{rem}

\subsection{Solutions of invariant differential equations}
We continue with a representation $\tGV$ of Capelli type.
Let $(\tG_\R : V_\R)$ be a real form of $\tGV$ in the sense
of \cite[\S4.1, Proposition~4.1]{Kim}, cf.~also
\cite[\S1.2]{Mu1} and \cite[\S4.2 \&~\S4.3]{Gyo}. Such a form
always exists. Let $M$ be a finitely generated $\DV$-module
of the form $\DV/I$.  Denote by $\calB_{V_\R}$ the space of
hyperfunctions on $V_\R$ (see, e.g., \cite[\S4.1]{Gyo}).
Then $\calB_{V_\R}$ is a $\DV$-module and the
``hyperfunction solutions space'' of $M$ is:
$$
\Sol(M,\calB_{V_\R}) = \{T \in \calB_{V_\R} : D.T= 0 \
\text{for all $D \in I$}\} \equiv
\Hom_{\cD_V}(M,\calB_{V_\R}).
$$
Notice that since any distribution on $V_\R$ can be viewed
as a hyperfunction, the ``distribution solutions space'' of
$M$ is contained in $\Sol(M,\calB_{V_\R})$.  An element $T$
of $\calB_{V_\R}$ such that $\tau(\g).T = 0$ is said to be
$\g$-invariant; it is called quasi-homogeneous if there
exist $t \in \N,\mu \in \C$ such that $(\Theta - \mu)^t.T =
0$.  Assume that $I = \DV\tau(\g) + \DV P$ for some $P \in
\DV$, then $\Sol(M,\calB_{V_\R})$ identifies with the space
of $\g$-invariant hyperfunctions $T$ which are solutions of
the equation $P.T= 0$.

The next corollary has been proved by M.~Muro
\cite[Theorem~4.1]{Mu4} for the real form $(\tG_\R=
\GL(n,\R), V_\R = \Sym_n(\R))$ of $\tGV = (\GL(n) :
S^2\C^n)$ when $P \ne 0$ and $\deg a_P = \ord P$. (See
Remark~\ref{remuro} for more details.)

\begin{cor}
  \label{cor58}
  Let $\tGV$ be of Capelli type and let $(\tG_\R : V_\R)$ be
  a real form of $\tGV$. Let $P \in \D[k]$ and write $P=
  b_D(\bTheta)Q_k + P_1$, $P_1 \in \Ker(\rad)$, as
  in~\eqref{eq51}. Assume that $b_D \neq 0$. Then,
  $\Sol(P,\calB_{V_\R})=\Sol(\cMP,\calB_{V_\R})= \{T \in
  \calB_{V_\R}: \text{$T$ $\g$-invariant, $P.T=0$}\}$ is
  finite dimensional and has a basis of quasi-homogeneous
  elements; it depends only on the polynomial $b_D(s)$ and
  the integer $k$.
\end{cor}

\begin{proof}
  We merely repeat the proof of M.~Muro (loc.~cit.).  A
  well-known result of M.~Kashiwara (see \cite[Th\'eor\`eme
  5.1.6]{Ka2}) says that if $M$ holonomic
  $\Sol(M,\calB_{V_\R})$ is a finite dimensional $\C$-vector
  space.  Therefore by combining the remarks above and
  Theorem~\ref{thm57} we obtain that $\calS=
  \Sol(P,\calB_{V_\R})$ is finite dimensional. Now observe
  that $[\Theta,P] = kP$ and $[\Theta,\tau(\g)] = 0$ imply
  that $\calS$ is stable under the action of $\Theta$.
  Therefore $\calS$ decomposes as a direct sum of spaces of
  the form $\Ker(\Theta - \mu \, \id_\calS)^t$. We have
  noticed after~\eqref{eq53} that $\cMP$ depends only on
  $b_D(s)$ and $k$, therefore it is also the case for $\calS
  = \Hom_{\cD_V}(\cMP,\calB_{V_\R})$.
\end{proof}

% \begin{remark}
%   Corollaries~\ref{cor5} and~\ref{cor6} were first proved
%   in the case $(\tG= \GL(n) : V = S^2\C^n)$ by Masakazu
%   Muro (2002).
% \end{remark}
 
\subsection{Regular holonomic modules}
Assume that $\tGV$ is of Capelli type and $\dim V \qmod G
=1$.

We filter $\DV$ by order of differential operators and we
set $\DV_j = \{D \in \DV : \ord D \le j\}$.  For sake of
completeness we now recall some known (or easy) results.

\begin{defn}
  \label{def59}
  Let $M$ be a finitely generated $\DV$-module.
  
  -- $M$ is {\em monodromic} if $\dim_\C \C[\Theta].x <
  \infty$ for all $x\in M$.

  -- $M$ is {\em homogeneous} if there exists a
  $\Theta$-stable good filtration on $M$, i.e.~a good
  filtration $FM = (F_pM)_{p \in \N}$ such that $\Theta.F_pM
  \subset F_pM$ for all $p$.
  
  -- $x \in M$ is {\em quasi-homogeneous} (of weight
  $\lambda$) if there exists $j \in \N$ and $\lambda \in \C$
  such that $(\Theta - \lambda)^j.x = 0$; we set $M_\lambda
  = \bigcup_{j \in \N} \Ker_M(\Theta - \lambda)^j$.
\end{defn}

Recall the following result proved in
\cite[Theorem~1.3]{Nang1} (which holds in the analytic
case).

\begin{thm}
  \label{thm510}
  Let $M$ be a homogeneous $\DV$-module. Then:
  \begin{enumerate}
  \item[{\rm (1)}] $M = \DV.E$, $E$ finite dimensional and
    generated by quasi-homogeneous elements;
  \item[{\rm (2)}] if $FM = (F_pM)_{p \in \N}$ is a
    $\Theta$-stable good filtration on $M$, the space $F_pM
    \cap M_\lambda$ is finite dimensional for all $p \in
    \N$, $\lambda \in \C$.
  \end{enumerate}
\end{thm}

Using this result it is not difficult to obtain the next
corollary.

\begin{cor}
  \label{cor511}
  Let $M$ be a finitely generated $\DV$-module. The
  following assertions are equivalent:
  \begin{enumerate}
  \item[{\rm (i)}] $M$ is monodromic;
  \item[{\rm (ii)}] $M$ is homogeneous;
  \item[{\rm (iii)}] $M = \DV.E$, $E$ finite dimensional
    such that $\Theta.E \subset E$.
  \end{enumerate}
\end{cor}

Recall that a holonomic $\DV$-module is regular if there
exists a good filtration $FM$ on $M$ such that the ideal
$\ann_{\C[T^*V]} \gr^{F}(M)$ is radical, see
\cite[Corollary~5.1.11]{KK}.  Denote by:

-- $\mod_{\tilde{\calC}}^{\mathrm{rh}}(\cD_V)$ the category
of regular holonomic whose characteristic variety is
contained in $\tcV$;

-- $\mod_{\tilde{\calC}}^{\Theta}(\cD_V)$ the category of
monodromic modules with characteristic variety contained in
$\tcV$.

Let $G_1$ be the simply connected cover of $G$ and set
$\tG_1 = G_1 \times C$ (recall that $C \cong \C^*$ is the
connected component of the centre of $\tG$). One has:
$\Lie(\tG_1) = \tfg = \g \boplus \fc$, $\fc = \C\zeta$,
where $\tau(\zeta) = \bTheta$ as above.  The group $\tG_1$
maps onto $G \times C$, and therefore onto $\tG= GC$. It
follows that $\tG_1$ and $G \times C$ act on $V$; the orbits
in $V$ then coincide with the $\tG$-orbits $O_1,\dots,O_t$.
The category of $\tG_1$-equivariant $\DV$-modules is denoted
by
$$
\mod^{\tG_1}(\Dv).
$$
Observe that if
$$
\mod^{G\times C}(\Dv)
$$
is the category of $(G\times C)$-equivariant $\DV$-modules,
any object in $\mod^{G \times C}(\Dv)$ can be naturally
considered as an object of $\mod^{\tG_1}(\Dv)$.  When $G$ is
simply connected, e.g.~$G= \SL(n)$, $\SL(n)\times \SL(n)$,
$\SL(n) \times \Sp(m)$ or $\mathrm{G}_2$, we have $\tG_1 = G
\times C$ and these two categories are the same.

\begin{lem}
  \label{lem512}
  Let $M$ be a finitely generated $\DV$-module. Then:
  $$
  \mathrm{(i)} \; M \in \mod^{\tG_1}(\Dv) \iff \mathrm{(ii)}
  \; M \in \mod_{\tilde{\calC}}^{\mathrm{rh}}(\cD_V) \limply
  \mathrm{(iii)} \; M \in
  \mod_{\tilde{\calC}}^{\Theta}(\cD_V).
  $$
  In particular, $\mod^{G\times C}(\Dv)=
  \mod_{\tilde{\calC}}^{\mathrm{rh}}(\cD_V)$ when $G$ is
  simply connected.
\end{lem}

\begin{proof}
  (i) $\imply$ (ii): By \cite[Theorem~12.11]{Bo}, or
  \cite[\S 5]{Hot2}, $M$ is regular holonomic.  Its
  characteristic variety $\ch M$ is therefore a $\tG$-stable
  subvariety of $T^*V$. Let $X$ be an irreducible component
  of $\ch M$; then $X$ is a Lagrangian conical closed
  $\tG$-stable subvariety of $T^*V$ and, if $\pi : T^*V \sto
  V$ is the natural projection, \cite[\S~5, Lemme~1]{Ka1},
  implies that $X =\overline{T_{\pi(X)^{\mathrm{reg}}}^*V}$.
  But it is easy to see that $\pi(X)^{\mathrm{reg}}$ (the
  smooth locus of $\pi(X)$) is equal to $O_j$ for some $1
  \le j \le t$.  Hence $X= \calC_j$ and $\ch M \subset
  \tcV$.
  
  (ii) $\imply$ (iii) and (i): (We mimic the proof of
  \cite[Proposition~1.6]{Nang1}.) Choose a good filtration
  such that $I(M)=\ann_{\C[T^*V]} \gr^{F}(M)$ is radical.
  Since $\ch M \subset \tcV$, the principal symbols $\alpha=
  \sigma(\bTheta)$ and $\sigma(\tau(\xi))$, $\xi \in \g$,
  belong to $I(M)$, that is to say $\sigma(\tau(\xi))
  \gr_j^F(M) =\alpha \gr_j^F(M) = (0) \subset
  \gr_{j+1}^F(M)$; in other words: $\bTheta.F_jM=
  \Theta.F_jM \subset F_jM$ and $\tau(\xi).F_jM \subset
  F_jM$.  In particular, $M$ is homogeneous,
  i.e.~monodromic.  Let $x \in M$. Since $\dim \C[\Theta].x
  < \infty$ there exist $j \in M$ and
  $\lambda_1,\dots,\lambda_l$ such that $x \in \sum_{i=1}^l
  F_jM \cap M_{\lambda_i}$. From $[\tau(\g),\Theta] =0$ and
  $\tau(\g).F_jM \subset F_jM$ it follows that
  $\tau(\g).F_jM \cap M_{\lambda_i} \subset F_jM \cap
  M_{\lambda_i}$; hence $U(\g).x$ is contained in the finite
  dimensional space $\sum_{i=1}^l F_jM \cap M_{\lambda_i}$,
  cf.~Theorem~\ref{thm510}. This shows that the action of
  $\ftg$ on $M$ given by the $\tau(\xi)$, $\xi \in \ftg$, is
  locally finite. The formula
  $$
  \Exp^{t\xi}.x = \exp(t\tau(\xi)).x= \sum_{k \ge 0}
  \frac{t^k}{k!} \tau(\xi)^k.x
  $$
  then yields a rational action of $\tG_1$ on $M$ whose
  differential is given by multiplication by the elements
  $\tau(\xi)$. It remains to check that this action is
  compatible with the action of $\tG$ on $\DV$, which is an
  easy exercise.
\end{proof}

Following \cite{Nang1,Nang2,Nang3,Nang4, Nang5} we want to
describe the category
$\mod_{\tilde{\calC}}^{\mathrm{rh}}(\Dv)$ in terms of a
category of modules over $U=R = \C[z,\delta,\theta]$. A
finitely generated $U$-module $N$ is called {\em monodromic}
if, for all $v \in N$, $\dim_\C \C[\theta].v < \infty$.
Denote by
$$
\mod^\theta(U)
$$
the category of monodromic modules. Observe that $N \in
\mod^\theta(U)$ decomposes as
$$
N= \bigoplus_{\lambda \in \C} N_\lambda, \quad N_\lambda =
\bigcup_{j \ge 0}\Ker_N(\theta- \lambda)^j.
$$
Recall that $[\theta,z] =z$, $[\theta,\delta] = -\delta$,
$z\delta = b^*(\theta)$, $\delta z = b(\theta)$ and that the
roots of $b(-s)$ are
$\lambda_0+1=1,\lambda_1+1,\dots,\lambda_{n-1}+1$,
cf.~Theorem~\ref{thm31}.  We then obtain:
\begin{itemize}
\item $z.N_\lambda \subset N_{\lambda +1}$,
  $\delta.N_\lambda \subset N_{\lambda -1}$;
  % \item $\z\delta_{N_\lambda}=
  %   b^*(\lambda)\id_{N_\lambda}$,
\item $\dim N_\lambda < \infty$ ($N$ is finitely generated);
\item $z\delta$, resp.~$\delta z$, is bijective on
  $N_\lambda$ if and only if $\lambda \ne -\lambda_j$,
  resp.~$\lambda \ne -(\lambda_j+1)$, therefore $z :
  N_\lambda \isomto N_{\lambda+1}$, $\delta : N_{\lambda+1}
  \isomto N_{\lambda}$ if $\lambda \ne -(\lambda_j +1)$.
\end{itemize}
From these properties it is easy to give a description of
the category $\mod^\theta(U)$ in terms of ``finite diagrams
of linear maps'' as in \cite{Nang1,Nang2,Nang3,Nang5}.

Let $M \in \mod^{G \times C}(\Dv)$. Since the differential
of the $G$-action on $M$ is given by $x\mapsto \xi.x$, $\xi
\in \g, x \in M$, one has:
\begin{equation}
  \label{eq54}
  \Phi M= M^G =\{x \in M : \tau(\g).x = 0\}.
\end{equation}
Therefore $\DVtauG.M=0$ and, via $\rad: \DVG/\DVtauG \isomto
U$ (Theorem~\ref{thm38}), $\Phi M$ can be considered as a
$U$-module by: $u.x= D.x$ if $u = \rad(D)$, $x \in M^G$.
Observe also that the isomorphism $\rad$ yields a natural
structure of right $U$-module on the module $\cN=
\DV/\DV\tau(\g)$ by: $\bar{a}.u = \overline{aD}$ if $\bar{a}
\in \cN$ and $u=\rad(D) \in U$. For $N \in \mod^\theta(U)$
we set:
\begin{equation}
  \label{eq55}
  \Psi N= \cN \otimes_U N.
\end{equation}
With these notation we have:

\begin{prop}
  \label{prop513} {\rm (1)} Let $M \in \mod^{G \times
    C}(\Dv)$ and $N \in \mod^\theta(U)$, then:
$$
\Phi M \in \mod^\theta(U), \quad \Psi N \in \mod^{G \times
  C}(\Dv), \quad \Phi\Psi N = N.
$$
{\rm (2)} Suppose that any $M \in \mod^{G \times C}(\Dv)$ is
generated by $M^G$ as a $\DV$-module. Then the categories
$\mod^{G \times C}(\Dv)$ and $\mod^\theta(U)$ are equivalent
via the functors $\Phi$ and $\Psi$. If furthermore $G$ is
simply connected, we obtain:
$\mod_{\tilde{\calC}}^{\mathrm{rh}}(\Dv) \equiv
\mod^\theta(U)$.
\end{prop}

\begin{proof}
  (1) From $G$ reductive and $M$ finitely generated, one
  deduces that the $\DVG$-module $M^G$ is finitely
  generated. Recall that $M$ is monodromic
  (Lemma~\ref{lem512}); since $\theta.x= \rad(\bTheta).x=
  \Theta.x$ it follows that $\Phi M$ is monodromic. Thus
  $\Phi N \in \mod^\theta(U)$.
  
  It is clear that $\Psi N$ is finitely generated over
  $\DV$.  The group $G$ acts naturally on $\DV$ and this
  action passes to $\cN$ (note that $\DV\tau(\g)$ is
  $G$-stable). One easily checks that one can endow $\cN
  \otimes_U N$ with a rational $G$-module structure by
  setting: $g.(\bar{a} \otimes_U x) =
  \overline{g.a}\otimes_U x$ for $\bar{a} \in \cN, g \in G,
  x \in N$. Notice that since $N$ is monodromic the group $C
  = \exp(\C \zeta)$ acts on $N$ by $\Exp^{t\zeta}.x=
  \exp(t\theta).x$. One can then verify that $C$ acts on
  $\cN \otimes_U N$ by: $\Exp^{t\zeta}.(\bar{a} \otimes_U
  x)= \Exp^{t\zeta}.\bar{a}\otimes_U \exp(t\theta).x$, $t
  \in \C$. One shows without difficulty that this $G\times
  C$-action is compatible with the $\tG$-action on $\DV$.
  Moreover, with the previous notation we get that:
  \begin{align*}
    \frac{d}{dt}_{\mid t=0}\Exp^{t\xi}.(\bar{a} \otimes_U x)
    & = \overline{[\tau(\xi),a]} \otimes_U x=
    \overline{\tau(\xi)a} \otimes_U x = \tau(\xi).(\bar{a}
    \otimes_U x),\\
    \frac{d}{dt}_{\mid t=0}\Exp^{t\zeta}.(\bar{a} \otimes_U
    x) &= \overline{[\bTheta,a]} \otimes_Ux + \bar{a}
    \otimes_U \theta.x= \overline{\bTheta a}\otimes_Ux -
    \overline{a\bTheta}\otimes_u x + \bar{a} \otimes_U
    \theta.x
    \\
    & = \bTheta.(\bar{a}\otimes_Ux) - \overline{a}\otimes_U
    \theta.x+ \bar{a} \otimes_U \theta.x \\
    &= \bTheta.(\bar{a}\otimes_Ux) =
    \tau(\zeta).(\bar{a}\otimes_Ux).
  \end{align*}
  This shows that $\Psi N \in \mod^{G \times C}(\Dv)$. The
  equality $(\Psi N)^G = \bar{1}\otimes_UN$ follows easily
  from the definition of the $G$-action on $\Psi N$, hence
  $\Phi\Psi N = N$.

  (2) Note that there is a surjective $(G\times
  C)$-equivariant morphism of $\DV$-modules $\mathtt{m} :
  \Psi\Phi M = \cN \otimes_U M^G \sto M$ given by
  $\mathtt{m}(\bar{a}\otimes_Ux)= a.x$. Set $L= \Ker
  \mathtt{m}$, hence $\mathtt{m} : (\Psi\Phi M)/ L \isomto
  M$.  Then $L \in \mod^{G \times C}(\Dv)$ is generated by
  $L^G$, by hypothesis, and we obtain ($G$ is reductive):
$$
M^G \cong \bigl(\Psi\Phi M/ L\bigr)^G= (\Psi\Phi M)^G/L^G =
(\Phi\Psi\Phi M)/ L^G = M^G/ L^G.
$$
This implies $L^G= (0)$, thus $L=0$ and $\Psi\Phi M \equiv
M$.  The last statement follows from Lemma~\ref{lem512}
\end{proof}

The previous proposition and work of P.~Nang lead to the
following:

\begin{conj}
  \label{conj514}
  Let $\tGV$ be of Capelli type with $\dim V \qmod G
  =1$. Then the categories $\mod^{G \times C}(\Dv)$ and
  $\mod^\theta(U)$ are equivalent via the functors $\Phi$
  and $\Psi$.
\end{conj}

By Proposition~\ref{prop513}, this conjecture is equivalent
to showing that $M = \DV M^G$ for all $M \in \mod^{G \times
  C}(\Dv)$.

P.~Nang \cite{Nang1, Nang3, Nang5} has proved
Conjecture~\ref{conj514} in the cases: $(\SO(n) \times \C^*
: \C^n)$, $(\GL(n) \times \SL(n) : \Mat_n(\C))$, $(\GL(2m) :
\bigwedge^{2}\C^{2m})$.  It would be interesting to obtain a
uniform proof in the eight cases where $\tGV$ is of Capelli
type and $\dim V \qmod G =1$ (see Appendix~\ref{A1}). As
observed above the category $\mod^\theta(U)$ has a nice
combinatorial description, which would give, when $G$ is
simply connected, a classification of the regular holonomic
modules on $V$ whose characteristic variety is contained in
$\tcV$.

%%%%%%%%%%%%%%%%%%%%
\clearpage
% \vfill \newpage

%%%%%%%%%%%%%%%%%%%%%

\appendix
\section{Irreducible MF representations}
\label{A1}

\vspace{0.8cm}

% \begin{table}
\begin{center}
  \begin{sideways}
    % \begin{sidewaystable}
    %   \centering
    \begin{tabular}{|l|c|c|c|c|c|}
      \hline
      &   $\tGV$ & $\deg f$ & $b(s)$ & Capelli & 
      com.~parabolic  \\
      \hline \hline
      (1) & $(\SO(n) \times \C^* : \C^n)$ & $2$ &
      $(s+1)(s+n/2)$ & yes & yes \\ 
      \hline
      (2) & $(\GL(n) : S^2\C^n)$ & $n$ & $\prod_{i=1}^n(s+(i+1)/2)$
      & yes  & yes\\ 
      \hline
      (3) & $(\GL(n) : \Wedge^2\C^n)$,  $n$ even & $n/2$  &
      $\prod_{i=1}^{n/2}(s+2i-1)$ 
      & yes  & yes\\ 
      \hline
      (4) &$(\GL(n) \times \SL(n) : \Mat_n(\C))$ & $n$ &
      $\prod_{i=1}^{n}(s+i)$ 
      & yes  & yes\\ 
      \hline
      (5) &$(\Sp(n)\times \GL(2)   : \Mat_{2n,2}(\C))$ & $2$ &
      $(s+1)(s+2n)$ 
      & yes  & no \\ 
      \hline
      (6) & $(\SO(7) \times \C^* : \mathrm{spin}=\C^8)$ & $2$
      & $(s+1)(s+4)$ 
      & yes  & no \\ 
      \hline
      (7) & $(\SO(9) \times \C^* : \mathrm{spin}=\C^{16})$ &
      $2$ & $(s+1)(s+8)$ 
      & no  & no \\ 
      \hline
      (8)  & $(\mathrm{G}_2 \times \C^* : \C^{7})$ & $2$ & $(s+1)(s+7/2)$
      & yes  & no \\ 
      \hline
      (9) & $(\mathrm{E}_6 \times \C^* : \C^{27})$ & $3$ & $(s+1)(s+5)(s+9)$
      & no  & yes \\ 
      \hline
      (10) & $(\GL(4) \times \Sp(2) : \Mat_{4}(\C))$ & $4$ &
      $(s+1)(s+2)(s+3)(s+4)$ 
      & yes  & yes \\ 
      \hline
      (3') & $(\GL(n) : \Wedge^2\C^n)$, $n$ odd & ---  & ---
      & yes  & no \\ 
      \hline
      (4') & $(\GL(n) \times \SL(m) : \Mat_{n,m}(\C)), n \ne m
      $ & --- & --- 
      & yes  & no \\ 
      \hline
      (11) & $(\Sp(n) \times \GL(1): \C^{2n})$ & --- & ---
      & yes  & no \\ 
      \hline
      (12) & $(\Sp(n) \times \GL(3): \Mat_{2n,3}(\C))$ & --- & ---
      & no  & no \\ 
      \hline
      (10') & $(\GL(n) \times \Sp(2) : \Mat_{n,4}(\C))$, $n
      \ne 4$ & --- & --- 
      & yes  & no \\ 
      \hline
      (13) & $(\SO(10) \times \C^* :
      \half\mathrm{spin}=\C^{16})$ & --- & --- 
      & yes  & no \\ 
      \hline
    \end{tabular}
    % \end{sidewaystable}
  \end{sideways}
\end{center}
% \caption{Irreducible MF representations}
% \label{table1}
% \end{table}

\vfill

\newpage

%%%%%%%%%%%%%%%%%%%%%%%%%%%%%%%%%%%%%%%%%%%%%%%%%%%%%%%%%%%%%

%%%%%%%%%%%%%%%%%%% fin %%%%%%%%%%%
}
%%%%%%%%%%%%%%%%%%%%%%%%%%%%%%%%%%%%%%%%%%%%%%%%%%%
\end{document}